\documentclass[12pt, twoside]{article}
\usepackage[margin=1.5in]{geometry}
\usepackage{amsmath}
\usepackage{amsfonts}
\usepackage{amsthm}
\usepackage{graphicx}
\usepackage{amssymb}
\usepackage[dvips]{epsfig}
\usepackage{subfig}
\usepackage{epstopdf}
\usepackage[hang,flushmargin]{footmisc} 
\usepackage{float}
\usepackage{comment}
\usepackage{abstract}
\usepackage{color}

\usepackage{fancyhdr}
\newcommand{\R}{\mathbb{R}}
\newcommand{\C}{\mathbb{C}}

\newcommand{\E}{\mathcal{E}}
\newcommand{\M}{\mathcal{M}}

\newcommand{\K}{\mathcal{K}}
\renewcommand{\S}{\mathcal{S}}

\numberwithin{equation}{section}

\let \eps \varepsilon

\let \aleq \lesssim
\let \ds \displaystyle
\let \e \eps
\let \om \omega
\let \p \partial
\let \la \lambda

\newtheorem{theorem}{Theorem}[section]
\newtheorem{assumption}[theorem]{Assumption}

\newtheorem{corollary}[theorem]{Corollary}

\newtheorem{remark}[theorem]{Remark}

\newtheorem{lemma}[theorem]{Lemma}

\begin{document}

\title{
Solitary Waves and Dynamics for Subcritical Perturbations of Energy Critical NLS
}

\author{Matt Coles and Stephen Gustafson
\\ \emph{\small{Department of Mathematics, University of British Columbia}} \\ \emph{\small{1984 Mathematics Road, Vancouver, British Columbia, Canada V6T 1Z2}}}

\maketitle

\begin{abstract}
\noindent We consider a perturbed energy critical focusing Nonlinear Schr\"odinger Equation in three dimensions. We construct solitary wave solutions for focusing subcritical perturbations as well as defocusing supercritical perturbations. 
The construction relies on the resolvent expansion, which is singular due to the presence of a resonance. 
Specializing to pure power focusing subcritical perturbations we demonstrate, via variational arguments, and for a certain range of powers, the existence of a ground state solitary wave, which is then shown to be the previously constructed solution. 
Finally, we present a dynamical theorem 
which characterizes the fate of radially-symmetric solutions whose initial data are below the action of the ground state. 
Such solutions will either scatter or blow-up in finite time depending on the sign of a certain function of their initial data.  
\let\thefootnote\relax\footnote{
\emph{2010 Mathematics Subject Classification}: 35Q55.}
\end{abstract}



\tableofcontents

\section{Introduction and Main Results }
\label{intro}

We consider here Nonlinear Schr\"odinger equations in three
space dimensions, of the form
\begin{align}\label{NLSeps}
  i \partial_t u = -\Delta u - |u|^4 u - \eps g(|u|^2) u,
\end{align}
for $u(x,t): \R^3 \times \R \to \C$, with $\eps$ a small, real parameter.
Nonlinear Schr\"odinger-type equations are well-known to have 
important applications to areas such as quantum mechanics and optics,
but have also been intensively studied as models of nonlinear dispersive
phenomena more generally \cite{Fibich, Sulem}. 

Solutions (of sufficient regularity and spatial decay) 
conserve the {\it mass}, and {\it energy}:
\[
  \M(u) = \frac{1}{2} \int_{\R^3} |u|^2 \; dx, \qquad
  \E_\eps(u) = \int_{\R^3} \left\{ \frac{1}{2} |\nabla u|^2
  - \frac{1}{6} |u|^6 - \frac{\eps}{2} G(|u|^2) \right\} \; dx 
\] 
where $G' = g$. We are particularly interested in the existence
(and dynamical implications) of {\it solitary wave} solutions 
\[
  u(x,t) = Q(x) e^{i \omega t}
\]
of~\eqref{NLSeps}.
We will consider only real-valued solitary wave profiles, $Q(x) \in \R$,
for which the corresponding stationary problem is 
\begin{align}\label{elliptic}
  -\Delta Q - Q^5 - \eps f(Q) + \omega Q = 0, 
  \qquad f(Q) = g(Q^2) Q. 
\end{align}
Since the perturbed solitary wave equation~\eqref{elliptic} is the
Euler-Lagrange equation for the {\it action}
\[
  \mathcal{S}_{\eps,\omega}(u) := \E_{\eps}(u) + \omega \M(u) \; , 
\]
the standard {\it Pohozaev relations} \cite[p.~553]{Evans}
give necessary conditions for existence of finite-action solutions
of~\eqref{elliptic}:
\begin{equation} \label{Poho}
\begin{split}
  &0 = \K_\eps(u) 
  := \frac{d}{d\mu} \mathcal{S}_{\eps,\omega} (T_\mu u  ) \bigg|_{\mu=1} \\
&\quad \quad \ \; \quad \quad  = \int |\nabla Q|^2 - \int Q^6 + \eps \int \left( 3F(Q) - \frac{3}{2} Q f(Q)) \right) \\  
  & 0 = \K^{(0)}_{\eps, \omega}(u) 
  := \frac{d}{d\mu} \mathcal{S}_{\eps,\omega} (S_\mu u  ) \bigg|_{\mu=1} 
  = \eps \int \left( 3 F(Q) - \frac{1}{2} Q f(Q) \right) - \omega \int Q^2 
\end{split}
\end{equation}
where 
\[
  (T_\mu u)(x) := \mu^{\frac{3}{2}} u (\mu x), \qquad
  (S_\mu u)(x) := \mu^{\frac{1}{2}} u (\mu x)
\]
are the scaling operators preserving, respectively,
the $L^2$ norm and the $L^6$ (and $\dot H^1$) norm,
and $F' = f$ (so $F(Q) = \frac{1}{2} G(Q^2)$).
 
The corresponding unperturbed ($\eps = 0$) problem, the 3D quintic equation
\begin{equation} \label{NLS0}
  i \partial_t u = -\Delta u - |u|^4 u,
\end{equation}
is {\it energy critical} in the sense that the scaling 
\[
  u(x,t) \mapsto u_\lambda(x,t) := \lambda^{1/2} u(\lambda x, \lambda^2 t)
\]
which preserves~\eqref{NLS0}, also leaves invariant its energy
\[
  \E_0(u) = \int_{\R^3} \left\{ \frac{1}{2} |\nabla u|^2
  - \frac{1}{6} |u|^6 \right\} \; dx, \qquad
  \E_0(u_\lambda(\cdot,t)) = \E_0(u(\cdot,\lambda^2 t)).
\]
One implication of energy criticality is that~\eqref{NLS0} fails to 
admit solitary waves with $\omega \not= 0$ 
-- as can be seen from~\eqref{Poho} --
but instead admits the {\it Aubin-Talenti} {\it static} solution
\begin{equation} \label{elliptic0}
  W(x)= \left( 1 + \frac{|x|^2}{3} \right)^{-1/2}, \quad \quad 
  \Delta W + W^5 = 0 , 
\end{equation}
whose slow spatial decay means it fails to lie in $L^2(\R^3)$, though 
it does fall in the {\it energy space}
\[
  W \not\in L^2(\R^3), \quad
  W \in \dot H^1(\R^3) = \{ u \in L^6(\R^3) \; | \; \| u \|_{\dot H^1} :=
  \| \nabla u \|_{L^2} < \infty \}.
\]  
By scaling invariance, 
$W_\mu := S_\mu W = \mu^{1/2} W(\mu x)$, for $\mu > 0$,
also satisfy~\eqref{elliptic0}, as do 
their negatives and spatial translates $\pm W_\mu(\cdot + a)$ ($a \in \R^3$).
These functions (and their multiples) are well-known to be the only
functions realizing the best constant appearing in the Sobolev inequality 
\cite{aubin,talenti}
\[
  \int_{\R^3} |u|^6 \leq C_3 \left( \int_{\R^3} |\nabla u|^2 \right)^3, \quad 
  C_3 = \frac{\int_{\R^3} W^6}{ \left(\int_{\R^3} |\nabla W|^2 \right)^3 }
  = \frac{1}{ \left(\int_{\R^3} W^6 \right)^2 },
\] 
where the last equality used $\int |\nabla W|^2 = \int W^6$
(as follows from~\eqref{Poho}). 
A closely related statement is that $W$, 
together with its scalings, negatives and spatial translates,
are the only minimizers of the energy under the Pohozaev constraint~\eqref{Poho}
with $\eps = \omega = 0$:
\begin{equation} \label{var0}
\begin{split}
  &\min\{ \E_0(u) \; | \; 0 \not= u \in \dot H^1(\R^3), \; \K_0(u) = 0 \}
  = \E_0(W) = \E_0( \pm W_\mu(\cdot + a)  ) , \\
  &\quad \K_0(u) = \int_{\R^3} \left\{ |\nabla u|^2 - |u|^6 \right\}.
\end{split}
\end{equation}
It follows that for solutions of~\eqref{NLS0} lying energetically 
`below' $W$, $\E_0(u) < \E_0(W)$, the sets where $\K_0(u) > 0$
and where $\K_0(u) < 0$ are invariant for~\eqref{NLS0}. 
Kenig-Merle \cite{Kenig} showed that radially symmetric solutions
in the first set scatter to $0$, while those in the second set
become singular in finite time (in dimensions 3, 4, 5).
In this way, $W$ plays a central role in classifying solutions of~\eqref{NLS0},
and it is natural to think of $W$ (together with its scalings and spatial translates) 
as the {\it ground states} of~\eqref{NLS0}.
The assumption in \cite{Kenig} that solutions be radially symmetric was removed in \cite{Killip_Visan} for dimensions $n \geq 5$ and then for $n=4$ in \cite{Dodson}. 
Removing the radial symmetry assumption appears still open for $n=3$. 
A characterization of the dynamics for initial data at the threshold $\mathcal{E}_0(u_0) = \mathcal{E}_0(W)$ appears in \cite{Duyc}, and a classification of global dynamics based on initial data slightly above the ground state is given in \cite{Naka}. 

Just as the main interest in studying~\eqref{NLS0}
is in exploring the implications of critical scaling, the main interest
in studying~\eqref{NLSeps} and~\eqref{elliptic} 
here is the effect of {\it perturbing} the critical scaling, in particular:
the emergence of ground state solitary waves from the static solution $W$,
the resulting energy landscape, and its implications for the dynamics. 

A natural analogue for~\eqref{elliptic} of the ground state variational 
problem~\eqref{var0} is
\begin{equation} \label{minS}
  \min \{ \mathcal{S}_{\eps,\omega}(u) \; |  \; u \in H^1 \setminus \{0\}, 
  \mathcal{K}_\eps(u) = 0 \}.
\end{equation} 
For a study of similar minimization problems see \cite{Bere2} and \cite{Bere3} as well as \cite{Alves}, which 
treats a large class of critical problems and establishes the existence of ground state solutions. 
In space dimensions $4$ and $5$, ~\cite{Slim,Slim2} 
showed the existence of minimizers for (the analogue of)~\eqref{minS}, hence of
ground state solitary waves, for each $\omega >0$ 
and $\eps g(|u|^2)u$ sufficiently small and subcritical;
moreover, a blow-up/scattering dichotomy
`below' the ground states in the spirit of~\cite{Kenig} holds.
Our main goal in this paper is to establish the existence of ground
states, and the blow-up/scattering dichotomy, in the $3$-dimensional setting.
In dimension $3$, the question of the existence of 
minimizers for~\eqref{minS} is more subtle,
and we proceed via a perturbative construction, rather than 
a direct variational method.

A key role in the analysis is played by the linearization of~\eqref{elliptic0}
around $W$, in particular the linearized operator
\begin{align*}
	H:= -\Delta +V :=-\Delta -5W^4,
\end{align*}
which as a consequence of scaling invariance has the following {\it resonance} :
\begin{equation} \label{res_un_norm}
   H \; \Lambda W =0, \quad \Lambda W := 
  \frac{d}{d \mu} S_{\mu} W |_{\mu=0} =  
  \left( \frac{1}{2} +x \cdot \nabla \right)W \;
  \notin L^2(\R^3).
\end{equation}
Indeed $\Lambda W = W^3 - \frac{1}{2} W$ decays like $|x|^{-1}$, and so 
\[
  W, \; \Lambda W \; \in L^r(\R^3) \cap \dot H^1(\R^3), \;\; 3 < r \leq \infty.
\] 

Our first goal is to find solutions to (\ref{elliptic}) where $\omega=\omega(\eps)>0$ is small and $Q(x) \in \R$ is a perturbation of $W$ in some appropriate sense. 
One obstacle is that $W \notin L^2$ is a slowly decaying function, 
whereas solutions of~\eqref{elliptic} satisfy $Q \in L^2$, 
and indeed are exponentially decaying.

\begin{assumption}\label{fassump}
Take $f:\R \to \R \in C^1$ such that $f(0)=0$ and 
\begin{align*}
	|f'(s)| \aleq |s|^{p_1-1} + |s|^{p_2-1}
\end{align*}
with $2< p_1 \leq p_2 <\infty$. 
Further assume that 
\begin{align*}
\langle \Lambda W, f(W) \rangle <0. 
\end{align*}
\end{assumption}
\begin{remark}
Here, as elsewhere in the paper, the bracket $\langle \cdot \; , \; \cdot \rangle$
denotes the standard $L^2(\R^3;\R)$ inner-product
\[
  \langle g , \; h \rangle = \int_{\R^3} g(x) h(x) \; dx
\]
interpreted where necessary as the usual $L^r$-$L^{r'}$ duality pairing.
For example, in Assumption~\ref{fassump}, 
$f(W) \in L^q$ for any $q > \frac{3}{p_1}  > \frac{3}{2}$, while 
$\Lambda W \in L^r$ for all $r > 3$. So by choosing $r$ close enough to $3$,
the inner product makes sense as an $L^r$-$L^{r'}$ pairing.
\end{remark}

\begin{theorem}\label{theorem1} 
There exists $\eps_0>0$ such that for each $0<\eps \leq \eps_0$, 
there is $\omega=\omega(\eps) > 0$, and smooth, real-valued, radially symmetric $Q=Q_\eps \in H^1(\R^3)$ satisfying \emph{(\ref{elliptic})} with 
\begin{align}
	\omega&=\omega_1 \eps^2 + \tilde{\omega} \label{wexp}\\
	Q(x) &= W(x) + \eta(x)  \label{Qexp}
\end{align}
where
\begin{align*}
\omega_1 = \left( \frac{-\langle \Lambda W, f(W) \rangle}{6 \pi} \right)^2,
\end{align*}
$\tilde{\omega} =O(\eps^{2+\delta_1})$ 
for any $\delta_1 < \min(1, p_1-2)$,
$\|\eta\|_{L^r} \aleq \eps^{1-3/r}$ for all $3 < r \leq \infty$, 
and $\|\eta\|_{\dot{H}^1} \aleq \eps^{1/2}$.  
In particular, $Q \rightarrow W$ in $L^{r} \cap \dot{H}^1$ as $\eps \to 0$.
\end{theorem}

\begin{remark}
We have a further decomposition of $\eta$ but the leading order term depends on whether we measure it in  $L^r$ with $r=\infty$ or $3<r<\infty$. See Lemmas \ref{etalemma} and \ref{etabound}. 
\end{remark}

\begin{remark} \label{nosup}
Note that allowable $f$ include $f(Q)=|Q|^{p-1}Q$
with $2<p<5$, the subcritical, pure-power, focusing nonlinearities, as well as 
$f(Q) = -|Q|^{p-1} Q$ with $5<p<\infty$, the supercritical, pure power, defocusing nonlinearities.  Observe
\begin{align*}
	\langle \Lambda W, W^p \rangle &= \int \left( \frac{1}{2} W^{p+1} + W^p (x \cdot \nabla) W \right) \\
	&= \int \left( \frac{1}{2} W^{p+1} + \frac{1}{p+1} (x \cdot \nabla) W^{p+1} \right) \\
	&= \int \left( \frac{1}{2} - \frac{3}{p+1} \right) W^{p+1} 
\end{align*}
which is negative when $2<p<5$ and positive when $p>5$. 
\end{remark}

\begin{remark}
Since $Q_\eps \to W$ in $L^r$ for $r \in (3,\infty]$,
the Pohozaev identity~\eqref{Poho}, together with the divergence theorem, implies that for any such family of solutions, a necessary condition is
\begin{align*}
  \langle \Lambda W, f(W) \rangle = 
  \int \left( \frac{1}{2} W f(W) - 3F(W) \right) 
  &= \lim_{\eps \to 0} \int \left( \frac{1}{2} Q_\eps f(Q_\eps) - 3 F(Q_\eps) \right) \\
  &\leq 0.
\end{align*}
\end{remark}

\begin{remark} \label{H1dot}
Note that $Q \in L^r \cap \dot{H}^1$ $(3<r \leq \infty)$ satisfying \emph{(\ref{elliptic})} lies automatically in $L^2$ (and hence $H^1$) by the Pohozaev relations~\eqref{Poho}:
\begin{align} \label{theotherone}
	0 = \int |\nabla Q|^2 -\int Q^6 - \eps \int f(Q) Q + \omega \int Q^2.
\end{align}
The first two integrals are then finite. 
We can also bound the third 
\begin{align*}
	\left| \int f(Q) Q \right| \leq \int |f(Q)| |Q| \aleq \int |Q|^{p_1+1} + \int |Q|^{p_2+1} < \infty
\end{align*}
since $p_2+1 \geq p_1+1 >3$. 
In this way $\int Q^2$ must be finite. 
Moreover, since $Q \in L^r$ with $r > 6$, a standard elliptic regularity argument
implies that $Q$ is in fact a smooth function.
Therefore it suffices to find a solution $Q \in L^r \cap \dot{H}^1$. 
\end{remark}

The paper \cite{Davila} considers an elliptic problem similar to (\ref{elliptic}):
\begin{align*}
-\Delta Q + Q - Q^p - \lambda Q^q = 0
\end{align*} 
with $1<q<3$, $\lambda>0$ large and fixed, and $p<5$ but $p \to 5$. 
They demonstrate the existence of three positive solutions, one of which approaches $W$ (\ref{elliptic0}) as $p \to 5$. 
The follow up \cite{Chen} established a similar result with $p \to 5$ but $p >5$ and $3<q<5$. While \cite{Davila} and \cite{Chen} are perturbative in nature, their method of construction differs from ours.

The proof of Theorem \ref{theorem1} is presented in Section~\ref{construction}.
As the statement suggests, the argument is perturbative -- 
the solitary wave profiles $Q$ are constructed as
small (in $L^r$) corrections to $W$. The set-up is given in Section~\ref{SetUp}.
The equation for the correction $\eta$ involves the resolvent of the linearized 
operator $H$. A Lyapunov-Schmidt-type procedure is used to recover
uniform boundedness of this resolvent in the presence of the resonance
$\Lambda W$ -- see Section~\ref{res} for the relevant estimates -- 
and to determine the frequency $\omega$, see Section~\ref{findinglambda}.
Finally, the correction $\eta$ is determined by a fixed point argument
in Section~\ref{eta}. 


The next question is if the solution $Q$ is a {\it ground state} in a suitable sense. 
For this question, we will specialize to pure, subcritical powers 
\[
  f(Q)=|Q|^{p-1} Q, \quad 3<p<5, 
\]
for which the `ground state' variational problem~\eqref{minS} reads
\begin{equation} \label{minSp}
\begin{split}
  &\qquad \min \{ \mathcal{S}_{\eps,\omega}(u) \; |  \; u \in H^1(\R^3) \setminus \{0\}, 
  \mathcal{K}_\eps(u) = 0 \}, \\
  & \mathcal{S}_{\eps,\omega}(u) = 
  \frac{1}{2} \|\nabla u\|_{L^2}^2 - \frac{1}{6} \|u\|_{L^6}^6 
  - \frac{1}{(p+1)} \eps \|u\|_{L^{p+1}}^{p+1} + \frac{1}{2} \omega \| u \|_{L^2}^2, \\
  &\mathcal{K}_{\eps}(u) = 
  \|\nabla u\|_{L^2}^2 - \|u\|_{L^6}^6 - \frac{3(p-1)}{2(p+1)} \eps \|u\|_{L^{p+1}}^{p+1}.\end{split}
\end{equation}

\begin{theorem}\label{theorem2} 	
Let $f(Q)=|Q|^{p-1} Q$ with $3<p<5$. 
There exists $\eps_0$ such that for each $0< \eps \leq \eps_0$ and 
$\omega = \omega(\eps) > 0$ furnished by Theorem~\ref{theorem1}, 
the solitary wave profile $Q_\eps$ constructed in Theorem \ref{theorem1}
is a minimizer of problem \eqref{minSp}. Moreover, $Q_\eps$
is the unique positive, radially-symmetric minimizer.
\end{theorem}

\begin{remark}
It follows from Theorem~\ref{theorem2} that the
solitary wave profiles are positive: $Q_\eps(x) > 0$.
\end{remark}

\begin{remark}\label{range_remark}
(see Corollary \ref{range}). 
By scaling, for each $\eps > 0$ there is an interval 
$[\underline{\omega}, \infty) \ni \omega(\eps)$,
such that for $\omega \in [\underline{\omega}, \infty)$,
\begin{align*}
 Q(x) := \left( \frac{\eps}{\hat\eps} \right)^{\frac{1}{5-p}} 
 Q_{\hat\eps} \left(  \left( \frac{\eps}{\hat\eps} \right)^{\frac{2}{5-p}}  x \right),
\end{align*} 
where $0 < \hat\eps \leq \eps_0$ satisfies 
$(\omega(\hat\eps) / \omega) = \left(\hat\eps / \eps \right)^{4/(5-p)} $,
solves the corresponding minimization problem \eqref{minSp}.
Here the function $Q_{\hat\eps}$ is the solution constructed by Theorem \ref{theorem1} with $\hat\eps$ and $\omega(\hat\eps)$. 
\end{remark}

The proof of Theorem~\ref{theorem2} is presented in Section~\ref{variational}.
It is somewhat indirect. We first use the $Q = Q_\eps$ constructed in Theorem~\ref{theorem1} simply as test functions to verify
\[
  \mathcal{S}_{\eps,\omega(\eps)}(Q_\eps) < \E_0(W)
\]  
and so confirm, by standard methods, that the variational 
problems~\eqref{minSp} indeed admit minimizers.
By exploiting the unperturbed variational problem~\eqref{var0}, we
show these minimizers approach (up to rescaling) $W$ as $\eps \to 0$.
Then the local uniqueness provided by the fixed-point argument
from Theorem~\ref{theorem1} implies that the minimizers 
agree with $Q_\eps$. 
We remark that uniqueness of the analogous ground states in 
high space dimensions is established in the recent preprint~\cite{Slim3}.  

Finally, as in~\cite{Slim,Slim2}, we use the variational problem~\eqref{minSp} 
to characterize the dynamics
of radially-symmetric solutions of the 
perturbed critical Nonlinear Schr\"odinger equation
\begin{equation} \label{NLS}
  \left\{ \begin{array}{c}
  i \partial_t u = -\Delta u - |u|^4 u - \eps |u|^{p-1} u  \\
  u(x,0) = u_0(x) \in H^1(\R^3)
  \end{array} \right.
\end{equation}
`below the ground state', in the spirit of \cite{Kenig}.
By standard local existence theory (details in Section~\ref{dynamics}),
the Cauchy problem~\eqref{NLS} admits a unique solution
$u \in C([0, T_{max}); H^1(\R^3))$ on a maximal time interval,
and central questions are whether this solution
{\it blows-up in finite time} ($T_{max} < \infty$) or is {\it global}
($T_{max} = \infty$), and if global, does it {\it scatter} (to $0$) in the sense
\[
  \lim_{t \to \infty} \| u(\cdot,t) - e^{i t \Delta} \phi_+ \|_{H^1} = 0 
\]
for some $\phi_+ \in H^1(\R^3)$. 
We have: 
\begin{theorem} \label{dichotomy}
Let $3 < p < 5$ and $0 < \eps < \eps_0$, let $u_0 \in H^1(\R^3)$
be radially-symmetric, and satisfy
\[
  S_{\eps,\omega(\eps)} (u_0) < S_{\eps,\omega(\eps)} (Q_\eps),
\]
and let $u$ be the corresponding solution to~\eqref{NLS}:
\begin{enumerate}
\item
If $K_{\eps}(u_0) \geq 0$, $u$ is global, and scatters to $0$ as $t \to \infty$; 
\item 
if $K_{\eps}(u_0) < 0$, $u$ blows-up in finite time .
\end{enumerate}
\end{theorem}
Note that the conclusion is sharp in the sense that $Q_\eps$ itself is
a global but {\it non-scattering} solution.
Below the action of the ground state the sets where $K_\eps(u)>0$ and $K_\eps(u)<0$ are invariant under the equation (\ref{NLSeps}).  
Despite the fact that $K_\eps(u_0) >0$ gives an a priori bound on the $H^1$ norm of the solution, the local existence theory is insufficient (since we have the energy critical power) to give global existence/scattering, and so we employ concentration compactness machinery.  

The blow-up argument is classical, while
the proof of the scattering result rests on that of~\cite{Kenig}
for the unperturbed problem, with adaptations to handle
the scaling-breaking perturbation coming from~\cite{Slim,Slim2}
(higher-dimensional case) and~\cite{KOPV} (defocusing case).
This is given in Section~\ref{dynamics}.


\section{Construction of Solitary Wave Profiles}
\label{construction}

This section is devoted to the proof of Theorem~\ref{theorem1},
constructing solitary wave profiles for the perturbed NLS via
perturbation from the unperturbed static solution $W$.

\subsection{Mathematical Setup} \label{SetUp}

Let $\lambda^2=\omega$ with $\lambda \geq 0$. 
Now substitute (\ref{Qexp}) to (\ref{elliptic}) to see 
\begin{align*}
	(-\Delta -5W^4 + \lambda^2) \eta = -\lambda^2 W + \eps f(W) + N(\eta)
\end{align*}
where 
\begin{align*}
	N(\eta)= (W+\eta)^5 - W^5 -5W^4 \eta + \eps\left( f(W+\eta) - f(W) \right)
\end{align*}
collects the terms which are (at least) quadratic in $\eta$,
or linear in $\eta$ with a factor of $\eps$. 
We can rewrite the above as
\begin{align}  \label{preinvert}
  (H + \lambda^2) \eta = \mathcal{F},
  \qquad H = -\Delta + V, \qquad V = - 5 W^4  
\end{align}
where 
\[
  \mathcal{F}=\mathcal{F}(\eps,\lambda,\eta) = 
  -\lambda^2 W + \eps f(W) + N(\eta) .
\]
To understand the resolvent 
$(H + \lambda^2)^{-1}$ 
for small $\lambda$, we follow~\cite{Kato}. Use the resolvent identity to write 
\begin{align*}
  (H + \lambda^2)^{-1} = (1 + R_0(-\lambda^2) V)^{-1} R_0(-\lambda^2)
\end{align*}
where 
\[
  R_0(\zeta)=(-\Delta -\zeta)^{-1}
\] 
is the free resolvent, and apply Lemma 4.3 of \cite{Kato} to obtain the expansion 
\begin{align}\label{expansionspirit}
  (1 + R_0(-\lambda^2) V)^{-1} = 
-  
  \frac{1}{\lambda} \langle V \psi, \cdot \rangle \psi + O(1) 
\end{align}
where $\psi$ is the normalized resonance eigenfunction~\eqref{res_un_norm}:
\begin{align} \label{Vpsi_int}
  \psi(x) = \frac{1}{\sqrt{3 \pi}} \Lambda W(x), \qquad 
  \int_{\R^3} V \psi = \sqrt{4 \pi}.
\end{align}  
The above expansion is understood in \cite{Kato} in weighted Sobolev spaces. We choose instead to work in higher $L^p$ spaces. 
Precise statements are found in the following Section~\ref{res}. 

To eliminate the singular behaviour as $\lambda \to 0$ we require 
\begin{align}\label{orth}
	0=\langle R_0(-\lambda^2) V \psi, \mathcal{F}(\eps,\lambda,\eta) \rangle . 
\end{align}
Satisfying this condition determines $\lambda=\lambda(\eps, \eta)$. 
This is done in Section ~\ref{findinglambda}. 
With this condition met, we can invert (\ref{preinvert}) to see 
\begin{align}\label{eta}
	\eta = (H+\lambda^2)^{-1} \mathcal{F} 
	= (H + \lambda^2(\eps,\eta))^{-1} \mathcal{F}(\eps,\lambda(\eps,\eta),\eta)
	 =: \mathcal{G}(\eta,\eps),
\end{align}
which can be solved for $\eta$ via a fixed point argument. 
This is done in Section~\ref{findingeta}. 

\subsection{Resolvent Estimates
}  \label{res}

We collect here some estimates that are necessary for the proof of Theorem \ref{theorem1}. 


In order to apply Lemma 4.3 of \cite{Kato} and so to use the expansion (\ref{expansionspirit}) 
in what follows (Lemmas \ref{expansion} and \ref{fullresolvent}) we must have that the operator $H$ has no zero eigenvalue.
However, it is true that $H (\partial W / \partial x_j)=0$ for each $j=1,2,3$. 
To this end, we restrict ourselves to considering only radial functions. 
In this way $H$ has no zero eigenvalues and only the one resonance, 
$\Lambda W$ (as follows, e.g., from \cite[Lemma 5.2]{Duyc}).

The free resolvent 
operator $R_0(-\lambda^2)$ for $\lambda>0$ has integral kernel
\begin{align}\label{kernel}
	R_0(-\lambda^2)(x)= \frac{e^{-\lambda|x|}}{4 \pi |x|}. 
\end{align}	
An application of Young's inequality/generalized Young's inequality  gives the bounds
\begin{alignat}{2}
&\|R_0(-\lambda^2) \|_{L^q \to L^r} &&\aleq  \lambda^{3(1/q-1/r)-2}, \quad 1\leq q \leq r \leq \infty
 \label{kernelbddnw} \\
&\|R_0(-\lambda^2) \|_{L_w^q \to L^r} &&\aleq  \lambda^{3(1/q-1/r)-2}, \quad 1 < q \leq r < \infty
 \label{kernelbddw}
\end{alignat}
with $3(1/q-1/r)<2$,  
as well as 
\begin{align}\label{kernelbdd}
	\|R_0(-\lambda^2)\|_{L^q \to L^r} \aleq 1
\end{align}
where $1<q<3/2$ and $3(1/q-1/r)=2$ (so $3<r<\infty$).
We will also need the additional bound 
\begin{align}\label{kernelinfty}
\|R_0(-\lambda^2)\|_{L^{\frac{3}{2}-} \cap L^{\frac{3}{2}+} \to L^\infty} \aleq 1, 
\end{align}
where the $+/-$ means the bound holds for any exponent greater/less than $3/2$, 
to replace the fact that we do not have (\ref{kernelbdd}) for $r=\infty$ and $q=3/2$.

Observe also that $R_0(0)=G_0$ has integral kernel
\begin{align*}
	G_0(x)=\frac{1}{4 \pi|x|}
\end{align*}
and is formally $(-\Delta)^{-1}$. 

We need also some facts about the operator $(1+R_0(-\lambda^2)V)^{-1}$. 
The idea is that we can think of the full resolvent $(1+R_0(-\lambda^2)V)^{-1} R_0(-\lambda^2)$ as behaving like
the free resolvent $R_0(-\lambda^2)$ provided we have a suitable orthogonality condition. 
Otherwise we lose a power of $\lambda$ due to the non-invertibility of $(1+G_0V)$: 
indeed,
\begin{equation} \label{kernels}
  \psi \in \ker(1 + G_0 V), \quad
  V \psi \in \ker\left( (1 + G_0 V)^*  = 1 + V G_0 \right) .
\end{equation}

First we recall some results of \cite{Kato}: 
\begin{lemma} \emph{(Lemmas 2.2 and 4.3 from \cite{Kato})} \label{expansion}
Let $s$ satisfy $3/2 < s < 5/2$ and denote $\mathcal{B} = B(H^{1}_{-s},H^{1}_{-s})$
where $H^1_{-s}$ is the weighted Sobolev space with norm
\begin{align*}
\|u\|_{H^1_{-s}} = \|(1+|x|^2)^{-s/2}u\|_{H^1}.
\end{align*} 
Then for $\zeta$ with $\emph{Im} \zeta \geq 0$ we have the expansions 
\begin{align*}
	1 + R_0(\zeta)V = 1 + G_0 V + i \zeta^{1/2} G_1 V + o(\zeta^{1/2}) \\
	\left( 1 + R_0(\zeta) V \right)^{-1} = - i \zeta^{-1/2} \langle \cdot , V \psi \rangle \psi + C^1_0 + o(1)
\end{align*}
in $\mathcal{B}$ with $|\zeta| \to 0$. 
Here $C^1_0$ is an explicit operator and $G_0$ and $G_1$ are convolution with the kernels 
\begin{align*}
G_0(x) = \frac{1}{4 \pi |x|}, \quad G_1(x) = \frac{1}{4 \pi}. 
\end{align*}
\end{lemma}
\begin{remark}
The expansion is also valid in $ B(L^2_{-s}, L^2_{-s})$ where $L^2_{-s}$ is the weighted $L^2$ space with norm
\begin{align*}
\|u\|_{L^2_{-s}} = \|(1+|x|^2)^{-s/2}u\|_{L^2}.
\end{align*}
\end{remark}
\begin{remark}
Since our potential only has decay 
$|V(x)| \aleq  \langle x \rangle^{-4}$ our expansion has one less term than in \emph{\cite{Kato}} and we use $3/2<s<5/2$ rather than $5/2<s<7/2$. 
\end{remark}

The following is a reformulation of Lemma \ref{expansion} but using higher $L^p$ spaces rather than weighted spaces. 
This reformulation was also used in \cite{Gust}. 
\begin{lemma} \label{fullresolvent}
Take $3<r\leq \infty$ and $\lambda>0$ small. Then
\begin{align*}
  \|(1 + R_0(-\lambda^2)V)^{-1} f \|_{L^r} \aleq \frac{1}{\lambda} \|f\|_{L^r}. 
\end{align*}
If we also have $\langle V \psi, f \rangle =0$ then
\begin{align*}
\|(1 + R_0(-\lambda^2)V)^{-1} f \|_{L^r} \aleq \|f\|_{L^r} 
\end{align*}
and 
\begin{align} \label{leading}
\| (1 + R_0(-\lambda^2)V)^{-1} f - \bar{Q} &(1+G_0 V)^{-1} \bar{P} f \|_{L^r}
\nonumber \\
&\lesssim
{\left\lbrace
\begin{array}{c}
\lambda^{1-3/r}, \; \ 3<r<\infty \\
\lambda \log(1/\lambda), \; \ r = \infty
\end{array}
\right\rbrace}
\|f\|_{L^r} 
\end{align}
where
\begin{align}
\begin{split} \label{projections}
P := \frac{1}{\int V \psi^2} \langle V \psi, \cdot \rangle \psi, \quad \bar{P} = 1-P \\
Q := \frac{1}{\int V \psi} \langle V, \cdot \rangle \psi, \quad \bar{Q} = 1-Q  
\end{split}
\end{align}
\end{lemma}
\begin{remark}
Since $V, \; V \psi \in L^1 \cap L^\infty$, the ``inner-product" 
$\langle V \psi, f \rangle$ makes sense
as a $L^{r'}$-$L^{r}$ duality pairing for $f \in L^r$, 
as do the actions of $P$ and $Q$ on $L^r$.
\end{remark}
\begin{remark}
The (non-self-adjoint) projections $P$ and $Q$ play a ``bookkeeping" role in~\eqref{leading}. Indeed, if $\langle V \psi, f \rangle = 0$, then 
simply $\bar{P} f = f$. The $\bar{P}$ is included to emphasize that 
$(1 + G_0 V)^{-1}$ is well-defined only on the subspace 
$L^r \cap (V \psi)^{\perp} = Ran \bar{P}$ (see proof below), and so any 
projection with $Ker P = (V \psi)^{\perp}$ would suffice. The particular choice
above is natural in the sense $Ran P = \langle \psi \rangle = Ker (1 + G_0 V)$.
Similarly, the $\bar{Q}$ indicates that 
$Ran(1 + G_0 V)^{-1} = L^r \cap V^{\perp} = Ran \bar{Q}$,
a consequence of estimate~\eqref{innerest} below.
\end{remark}
\begin{proof}
We start with the identity 
\begin{align*}
g :=(1 + R_0(-\lambda^2)V)^{-1} f 
&= f - R_0(-\lambda^2)V (1 + R_0(-\lambda^2)V)^{-1} f \\
&= f - R_0(-\lambda^2)V g 
\end{align*}
so 
\begin{align*}
\| g \|_{L^r} 
&\aleq \| f \|_{L^r} + \|R_0(-\lambda^2)V g \|_{L^r}.
\end{align*}
We treat the above second term in two cases. For $3<r<\infty$ let $1/q=1/r+2/3$ and use (\ref{kernelbdd}) and 
for $r=\infty$ use (\ref{kernelinfty}) 
\begin{align*}
\|R_0(-\lambda^2)V g \|_{L^r}
& \aleq  \left\{ 
\begin{array}{l} 
\|V g \|_{L^q}, \quad 3<r<\infty \\
\|V g \|_{L^{3/2^-} \cap L^{3/2^+}}, \quad r=\infty
\end{array}
\right. \\
& \aleq  \left\{ 
\begin{array}{l} 
\|V\langle x \rangle^{2}\|_{L^m} \|  g \|_{L^2_{-2}}, \quad 3<r<\infty \\
\|V\langle x \rangle^{2}\|_{L^{6^-} \cap L^{6^+}} \|  g \|_{L^2_{-2}}, \quad r=\infty 
\end{array}
\right. \\
&\aleq \| g \|_{L^2_{-2}}. 
\end{align*}
Here we used that 
$|V(x)| \aleq \langle x \rangle^{-4}$, and with 
$1/q=1/m+1/2$ we have $(4-2)m>3$. 
Finally we appeal to Lemma \ref{expansion} and use the fact that $L^r \subset L^{2}_{-2}$ to see 
\begin{align*}
\|R_0(-\lambda^2)V g \|_{L^r} \aleq \|(1 + R_0(-\lambda^2)V)^{-1} f \|_{L^2_{-2}} 
\aleq \frac{1}{\lambda} \|f\|_{L^2_{-2}}
\aleq \frac{1}{\lambda} \|f\|_{L^r}
\end{align*}
where we can remove the factor of $1/\lambda$ if our orthogonality condition is satisfied. 

In light of~\eqref{kernels},
\[
  1 + G_0 V : L^r \cap V^{\perp} \to L^r \cap (V \psi)^{\perp}
\]
is bijective, and so we treat the operator $(1+G_0V)^{-1}$ as acting
\[
  (1 + G_0 V)^{-1} : L^r \cap (V\psi)^{\perp} \rightarrow L^r \cap V^{\perp},
\]  
which is the meaning of the expression 
$\bar{Q} (1 + G_0 V)^{-1} \bar{P}$ involving the projections
$\bar{P}$ and $\bar{Q}$.
That the range should be taken to be $V^{\perp}$ is a consequence
of estimate~\eqref{innerest} below. 

To prove~\eqref{leading}, expand
\[
\begin{split}
  & R_0(-\lambda^2) = G_0 - \lambda G_1 + \lambda^2 \tilde{R}, \\
  & \tilde{R} := \frac{1}{\lambda^2} \left( R_0(-\lambda^2) - G_0 + \lambda G_1 \right)
   = \frac{1}{\lambda} \left( \frac{e^{-\lambda|x|} - 1 + \lambda|x|}{4 \pi \lambda|x|}
   \right) \ast
\end{split}  
\]
and consider $f \in (V \psi)^{\perp} \cap L^r$ with $3<r\leq \infty$. 
We first establish the estimates
\begin{align}\label{hest}
&\|h\|_{L^q}
\lesssim 
\begin{cases}
1, \quad 1<q<\infty\\
\log(1/\lambda), \quad q=1
\end{cases},
\qquad h := V \tilde{R} V \psi ,
 \\ 
&|\langle V, (1+R_0(-\lambda^2)V)^{-1}f \rangle | \lesssim 
{\left\lbrace
\begin{array}{c}
\lambda, \quad 3<r<\infty \\
\lambda \log(1/\lambda),\quad r = \infty
\end{array}
\right\rbrace}
\|f\|_{L^r} \label{innerest}
\end{align}
where, recall, the inner product in~\eqref{innerest} is interpreted
as an $L^{r'}$-$L^r$ duality pairing.
For the purpose of estimate~\eqref{hest}, we may make the following replacements: 
$V\psi \to \langle x \rangle^{-5}$, 
$V \to \langle x \rangle^{-4}$, 
and $\tilde{R}(x) \to \min(|x|,1/\lambda)$. 
To establish ($\ref{hest}$) we must therefore estimate 
\begin{align*}
\langle x \rangle^{-4} \int_{\R^3} \min(|y|, 1/\lambda) \langle y-x \rangle^{-5} dy,  
\end{align*}
and we proceed in two parts: 
\begin{itemize}
\item Take $|y|\leq 2|x|$. Then
\begin{align*}
\langle x \rangle^{-4} \int_{|y|\leq 2|x|} &\min(|y|, 1/\lambda) \langle y-x \rangle^{-5} dy \\
&\lesssim \langle x \rangle^{-4} \min(|x|, 1/\lambda)\int  \langle y-x \rangle^{-5} dy \\
&\lesssim \langle x \rangle^{-4} \min(|x|, 1/\lambda)
\end{align*}
and 
\begin{align*}
\|\langle x \rangle^{-4} &\min(|x|, 1/\lambda)\|_{L^q}^q \\
&\lesssim \int_0^1 r^{q+2} dr + \int_1^{1/\lambda} r^{-3q+2} dr + \frac{1}{\lambda} \int_{1/\lambda}^\infty r^{-4q+2}dr\\
&\lesssim 1 + 
{\left\lbrace
\begin{array}{c}
1, \quad q>1 \\
\log(1/\lambda), \quad q=1
\end{array}
\right\rbrace}
+ \lambda^{4(q-1)} \\
&\lesssim 
\begin{cases}
1, \quad q>1 \\
\log(1/\lambda), \quad q=1
\end{cases}
.
\end{align*}
\item Take $|y|\geq 2|x|$. Then 
\begin{align*}
\langle x \rangle^{-4} \int_{|y|\geq 2|x|} \min(|y|, 1/\lambda) \langle y-x \rangle^{-5} dy 
\lesssim
\langle x \rangle^{-4} \int |y| \langle y \rangle^{-5} dy 
\lesssim
\langle x \rangle^{-4}
\end{align*}
and 
\begin{align*}
\|\langle x \rangle^{-4}\|_{L^q} \lesssim 1 . 
\end{align*}
\end{itemize}

With (\ref{hest}) established we now prove (\ref{innerest}). 
Let $g= (1 + R_0(-\lambda^2)V)^{-1}f$ and observe 
\begin{align*}
0 &= \frac{1}{\lambda} \langle V \psi , f \rangle \\
&= \frac{1}{\lambda} \langle V \psi, (1 + R_0(-\lambda^2)V)g \rangle \\
&= \frac{1}{\lambda} \langle (1 + VR_0(-\lambda^2))(V \psi), g \rangle \\
&= \frac{1}{\lambda} \langle (1 + V(G_0-\lambda G_1 + \lambda^2 \tilde{R}) )(V \psi), g \rangle \\
&= \langle (  -VG_1 + \lambda V \tilde{R} )(V \psi), g \rangle \\
&= -\frac{1}{\sqrt{4\pi}} \langle V, g \rangle + \lambda \langle h, g \rangle 
\end{align*}
noting that $(1+VG_0)(V\psi)=0$ and using (\ref{Vpsi_int}). Now
\begin{align*}
|\langle V, g \rangle| \lesssim \lambda \|h\|_{L^{r'}} \|g\|_{L^r}
\lesssim \lambda 
{\left\lbrace
\begin{array}{c}
1, \quad 3<r<\infty \\
\log(1/\lambda), \quad r=\infty
\end{array}
\right\rbrace}
\|f\|_{L^r}
\end{align*}
applying (\ref{hest}).

With (\ref{innerest}) in place we finish the argument. For $f \in L^r \cap (V\psi)^{\perp}$
we write
\begin{align*}
g = (1+R_0(-\lambda^2)V)^{-1} f \quad \text{ and } \quad g_0 = (1 + G_0V)^{-1} f.
\end{align*}
We have
\begin{align*}
0 = (1 + R_0(-\lambda^2)V) g - (1 + G_0V)g_0 
\end{align*}
and so 
\begin{align*}
(1+G_0V)(g-g_0) = - \hat R Vg
\end{align*}
where $\hat R = R_0(-\lambda^2) - G_0$. 
The above also implies $\hat R Vg \perp V \psi$. 
We invert to see
\begin{align*}
g-g_0 = - (1+ G_0 V)^{-1} \hat R Vg + \alpha \psi
\end{align*}
noting that $\psi \in \ker(1+G_0V)$. Take now inner product with $V$ to see
\begin{align*}
\alpha \langle V, \psi \rangle = \langle V, g \rangle
\end{align*}
and so 
\begin{align*}
|\alpha| \lesssim |\langle V, g \rangle |
\lesssim 
{\left\lbrace
\begin{array}{c}
\lambda, \quad 3<r<\infty \\
\lambda \log(1/\lambda),\quad r = \infty
\end{array}
\right\rbrace}
\|f\|_{L^r}
\end{align*}
observing (\ref{innerest}). 
It remains to estimate 
$(1+G_0V)^{-1} \hat R Vg$. 
We note that 
\begin{align*}
\hat R = \left( \frac{ e^{-\lambda|x|}-1}{4 \pi |x|} \right) \ast
\end{align*}
and so for estimates we may replace $\hat R(x)$ with 
$\min ( \lambda, 1/|x| )$. 
There follows 
by Young's inequality 
\begin{align*}
\|(1+G_0V)^{-1} \hat R Vg\|_{L^r}
&\lesssim \|\hat R Vg\|_{L^r} \\
&\lesssim \|\min ( \lambda, 1/|x| ) \|_{L^r} \|Vg\|_{L^1} \\
&\lesssim \|\min ( \lambda, 1/|x| ) \|_{L^r} \|g\|_{L^r} \\
& \lesssim \lambda^{1-3/r} \|f\|_{L^r}. 
\end{align*}
And so after putting everything together we obtain~\eqref{leading}.
\end{proof}

We end this section by recording pointwise 
estimates of the nonlinear terms 
\begin{align*}
	N(\eta) &= (W+\eta)^5 - W^5 -5W^4 \eta + \eps\left( f(W+\eta) - f(W) \right) .\\
\end{align*}
Bound the first three terms as follows: 
\begin{align*}
|(W+\eta)^5 - W^5 -5W^4 \eta | \aleq  W^3 \eta^2 +  |\eta|^5 .
\end{align*}
For the other term we use the Fundamental Theorem of Calculus 
and Assumption~\ref{fassump} 
to see
\begin{align*} 
|f(W+\eta)-f(W)| 
&= \left| \int_0^1 \partial_\delta f(W+\delta \eta) d \delta \right|  \\
&=\left| \int_0^1  f'(W+\delta \eta) \eta d \delta \right|  \\
& \aleq  |\eta| \sup_{0 < \delta <1} \left(  |W+\delta \eta|^{p_1-1}  + |W+\delta \eta|^{p_2-1} \right)\\
& \aleq  |\eta| \left( W^{p_1-1} + |\eta|^{p_1-1} + W^{p_2-1} + |\eta|^{p_2-1} \right) \\
& \aleq  |\eta| \left(W^{p_1-1} + W^{p_2-1} \right) +  |\eta|^{p_1} +  |\eta|^{p_2}
\end{align*}
and so together we have 
\begin{align} \label{nonbdd}
|N(\eta)| \aleq
 W^3 \eta^2+|\eta|^5 
+ \eps |\eta| \left(W^{p_1-1} + W^{p_2-1} \right) +  \eps |\eta|^{p_1} + \eps |\eta|^{p_2}. 
\end{align}
Similarly 
\begin{align*}
	&|f(W
	+\eta_1) - f(W+ \eta_2)|  \\
	& \quad \quad \aleq |\eta_1-\eta_2| 
\left( W^{p_1-1} + |\eta_1|^{p_1-1} + |\eta_2|^{p_1-1} + W^{p_2-1} 
+ |\eta_1|^{p_2-1} + |\eta_2|^{p_2-1} \right) 
\end{align*}
and so
\begin{align} \label{nondiff}
|N(\eta_1) - N(\eta_2)| & \aleq  
 |\eta_1-\eta_2| (|\eta_1|+|\eta_2|) W^3 +  |\eta_1-\eta_2| (|\eta_1|^4+|\eta_2|^4)
\nonumber \\
& \quad + \eps |\eta_1 - \eta_2| \left( W^{p_1-1} + W^{p_2-1} \right) \nonumber \\
& 
\quad + \eps |\eta_1 - \eta_2| \left( |\eta_1|^{p_1-1} + |\eta_2|^{p_1-1} + 
|\eta_1|^{p_2-1} + |\eta_2|^{p_2-1} \right)
.
\end{align}

\subsection{Solving for the Frequency
}\label{findinglambda}

We are now in a position to construct solutions to (\ref{elliptic}) and so prove Theorem \ref{theorem1}. 
The proof proceeds in two steps. 
In the present section, we will solve for $\lambda$ in (\ref{orth}) 
for a given small $\eta$. 
Then in the following Section~\ref{eta},
we will treat $\lambda$ as a function of $\eta$ and solve (\ref{eta}). 
Both steps involve fixed point arguments.

We begin by computing the inner product (\ref{orth}). Write 
\begin{align*}
	0=\langle R_0(-\lambda^2) V \psi,\mathcal{F} \rangle
	=\langle R_0(-\lambda^2) V \psi,-\lambda^2 W + \eps f(W)  + N(\eta) \rangle
\end{align*}
so that
\begin{align}\label{lambda}
	\lambda \cdot \lambda \langle R_0(-\lambda^2) V \psi, W \rangle=
	\eps \langle R_0(-\lambda^2) V \psi,f(W) \rangle
	+\langle R_0(-\lambda^2) V \psi,N(\eta) \rangle. 
\end{align}
It is our intention to find a solution $\lambda$ of (\ref{lambda}) of the appropriate size. 
This is done in Lemma \ref{lambdalemma} but we first make some estimates on the leading order inner products appearing above. 

\begin{lemma}

We have the estimates 
\begin{align}
&\langle R_0(-\lambda^2) V \psi,f(W) \rangle =  -\langle \psi, f(W) \rangle + O( \lambda^{\delta_1}) \label{inner1} \\
&\lambda \langle R_0(-\lambda^2) V \psi, W \rangle = 2\sqrt{3 \pi} + O(\lambda) \label{inner2} 
\end{align}
where $\delta_1$ is defined in the statement of Theorem \ref{theorem1}. 
\end{lemma}

\begin{proof}

Firstly
\begin{align*}
\langle R_0(-\lambda^2) V \psi,f(W) \rangle = \langle G_0 V \psi,f(W) \rangle + \langle (R_0(-\lambda^2)-R_0(0)) V \psi,f(W) \rangle
\end{align*}
First note that since $H \psi =0$ we have $V \psi = -(-\Delta \psi)$ so 
\begin{align*}
\langle G_0 V \psi,f(W) \rangle = \langle -(-\Delta)^{-1}(-\Delta \psi),f(W) \rangle = -\langle \psi,f(W) \rangle.
\end{align*}
Note that this inner product is finite. 
For the other term use the resolvent identity $R_0(-\lambda^2) - R_0(0)= -\lambda^2 R_0(-\lambda^2) R_0(0) $
to see
\begin{align*}
\langle (R_0(-\lambda^2)-R_0(0)) V \psi,f(W) \rangle = \lambda^2 \langle R_0(-\lambda^2) \psi, f(W) \rangle .
\end{align*}
Observe now that 
\begin{align*}
\lambda^2 | \langle R_0(-\lambda^2) \psi, f(W) \rangle | 
\leq \lambda^2 \| R_0(-\lambda^2) \psi\|_{L^r} \|f(W)\|_{L^{r^*}}
\end{align*}
where $1/r+1/r^*=1$. 
Choose an $r^*>1$ with $3/p_1 < r^* <3/2$. In this way $f(W) \in L^{r^*}$ observing Assumption \ref{fassump}. 
We now apply (\ref{kernelbddw}) with $q=3$ noting that $3 < r < \infty$. Hence
\begin{align*}
\lambda^2 | \langle R_0(-\lambda^2) \psi, f(W) \rangle | 
&\aleq \lambda^2 \cdot \lambda^{3(1/3-1/r)-2} \|\psi\|_{L_w^3} \|f(W)\|_{L^{r^*}}\\
&\aleq \lambda^{1-3/r}. 
\end{align*}
If $p_1 \geq 3$ we can take $r$ as large as we like. 
Otherwise we must take $3 < r < 3/(3-p_1)$ and so $1-3/r$ can be made close to $p_1-2$ (from below). 
We now see (\ref{inner1}). 

Next on to (\ref{inner2}). Note that this computation is taken from~\cite{Gust}. 
First we isolate the troublesome part of $W$ and write 
\[
  W = \frac{\sqrt{3}}{|x|}  + \tilde{W}. 
\]
There is no problem with the second term since 
$\tilde{W} \in L^{6/5}$ and $V\psi \in L^{6/5}$ so we can use (\ref{kernelbdd}) with $q=6/5$ and $r=6$ to see 
\begin{align} 
\lambda |\langle R_0(-\lambda^2) V \psi, \tilde{W} \rangle | 
&\aleq \lambda \|R_0(-\lambda^2)V\psi\|_{L^6} \|\tilde{W}\|_{L^{6/5}} \nonumber \\
& \aleq \lambda \|V \psi\|_{L^{6/5}} \|\tilde{W}\|_{L^{6/5}} \\
& \aleq \lambda. \label{Wtilde}
\end{align}
Set $g := V \psi$ and 
concentrate on 
\begin{align*}
\lambda \sqrt{3} \left\langle R_0(-\lambda^2) g ,  \frac{1}{|x|} \right\rangle 
= \sqrt{\frac{6}{\pi}} \lambda \left\langle \frac{\hat{g}(\xi)}{ |\xi|^2 + \lambda^2} , \frac{1}{|\xi|^2} \right\rangle
\end{align*}
where we work on the Fourier Transform side,
using Plancherel's theorem. So
\begin{align*}
& \sqrt{\frac{6}{\pi}} \lambda \left\langle \frac{\hat{g}(\xi)}{ |\xi|^2 + \lambda^2} , \frac{1}{|\xi|^2} \right\rangle \\
& \quad \quad \quad =  \sqrt{\frac{6}{\pi}} \lambda \; \hat{g}(0) \left\langle \frac{1}{|\xi|^2 + \lambda^2} , \frac{1}{|\xi|^2} \right\rangle
+  \sqrt{\frac{6}{\pi}} \lambda \left\langle \frac{\hat{g}(\xi) - \hat{g}(0)}{|\xi|^2 + \lambda^2} , \frac{1}{|\xi|^2} \right\rangle
\end{align*}
where the first term is the leading order. We invert the Fourier Transform and note that 
$\hat{g}(0) = (2 \pi)^{-\frac{3}{2}} \int g$ to see
\begin{align*}
 \sqrt{\frac{6}{\pi}} \lambda \; \hat{g}(0) \left\langle \frac{1}{|\xi|^2 + \lambda^2} , \frac{1}{|\xi|^2} \right\rangle
&= \sqrt{3} \left( \int g \right) \lambda
\left\langle \frac{e^{-\lambda|x|}}{4 \pi |x|}, \frac{1}{|x|} \right\rangle \\
&= \sqrt{3} \int g = 2 \sqrt{3 \pi}.
\end{align*}
We now must bound the remainder term. It is easy for the high frequencies 
\begin{align*}
\int_{|\xi| \geq 1} \frac{|\hat{g}(\xi) - \hat{g}(0)|}{|\xi|^2(|\xi|^2 + \lambda^2)} d\xi
\aleq \|\hat{g}\|_{L^\infty} \int_{|\xi| \geq 1} \frac{d \xi}{|\xi|^4} \aleq \|g\|_{L^1} \aleq 1. 
\end{align*}
For the low frequencies note that since $|x|g \in L^{1}$ we have that $\nabla \hat{g}$ is continuous and bounded. 
In light of this set 
\begin{align*}
h(\xi) := \phi(\xi) \left( \hat{g}(\xi) - \hat{g}(0) - \nabla \hat{g}(0) \cdot \xi \right)
\end{align*}
where $\phi$ is a smooth, compactly supported cutoff function with $\phi = 1$ on $|\xi| \leq 1$. 
Now since
\begin{align*}
\int_{|\xi|\leq 1} \frac{\xi}{|\xi|^2 (|\xi|^2 + \lambda^2)} d \xi = 0 
\end{align*}
we have 
\begin{align*}
\int_{|\xi| \leq 1} \frac{\hat{g}(\xi) - \hat{g}(0)}{|\xi|^2(|\xi|^2 + \lambda^2)} d\xi
= \int_{|\xi| \leq 1} \frac{h(\xi)}{|\xi|^2(|\xi|^2 + \lambda^2)} d\xi
\end{align*}
and so bound this integral instead. If we recall the form of $g$ we see $|g| \aleq \langle x \rangle^{-5}$ and so $(1+|x|^{1+\alpha})g \in L^1$ for some $\alpha>0$. Therefore $(1+|x|^{1+\alpha}) \check{h} \in L^1$ and noting also that $\nabla h(0) = 0$ we see $|\nabla h(\xi)| \aleq \text{min}(1,|\xi|^{\alpha})$.
The Mean Value Theorem along with $h(0)=0$ then gives $|h(\xi)| \aleq \text{min}(1,|\xi|^{1+\alpha})$. With this bound established we consider two regions of the integral
\begin{align*}
\int_{|\xi|\leq \lambda}\frac{|h(\xi)|}{|\xi|^2(|\xi|^2 + \lambda^2)} d\xi 
\aleq 
\int_{|\xi|\leq \lambda} \frac{|\xi|}{|\xi|^2(|\xi|^2 + \lambda^2)} d\xi 
\aleq
\int_{|\zeta|\leq 1} \frac{1}{|\zeta|(|\zeta|^2 + 1)} d\zeta 
\aleq 1
\end{align*}
and 
\begin{align*}
\int_{\lambda \leq |\xi|\leq 1}\frac{|h(\xi)|}{|\xi|^2(|\xi|^2 + \lambda^2)} d\xi 
&\aleq 
\int_{\lambda \leq |\xi|\leq 1} \frac{|\xi|^{1+\alpha}}{|\xi|^2(|\xi|^2 + \lambda^2)} d\xi \\
&\aleq
\lambda^{\alpha} \int_{1 \leq |\zeta|\leq 1/\lambda} \frac{|\zeta|^{\alpha-1}}{|\zeta|^2 + 1} d\zeta 
\aleq \lambda^{\alpha} \cdot \lambda^{-\alpha}
\aleq 1.
\end{align*}
Putting everything together gives (\ref{inner2}). 
\end{proof}

With the above estimates in hand we turn our attention to solving (\ref{lambda}). 

\begin{lemma}\label{lambdalemma}
For any $R>0$ there exists $\eps_0 = \eps_0(R)>0$ such that for $0<\eps \leq \eps_0$ and given a fixed $\eta \in L^\infty$ with $\|\eta\|_{L^\infty} \leq R \eps$ the equation 
\emph{(\ref{lambda})} has a unique 
solution 
$\lambda = \lambda(\eps, \eta)$ satisfying $\eps \lambda^{(1)} /2 \leq \lambda \leq 3 \eps \lambda^{(1)}/2$
where
\begin{align}\label{lambda1}
	\lambda^{(1)}= \frac{-\langle \Lambda W , f(W) \rangle}{6 \pi}>0. 
\end{align}
Moreover, we have the expansion 
\begin{align}\label{lambdaexp}
	\lambda = \lambda^{(1)} \eps + \tilde{\lambda},
       \qquad \tilde{\lambda}=O(\eps^{1+\delta_1}).
\end{align} 
\end{lemma}

\begin{remark}
Writing the resolvent as \emph{(\ref{kernel})}, and thus the subsequent estimates \emph{(\ref{kernelbddnw})}-\emph{(\ref{kernelbdd})}, require $\lambda>0$ and so it is essential that we have established $\lambda^{(1)}>0$.  
This is the source of the sign condition in Assumption \ref{fassump}.
\end{remark}

\begin{proof}
We first estimate the remainder term. 
Take $  \eps \lambda^{(1)}/2 \leq \lambda \leq 3 \eps \lambda^{(1)} /2$ and $\eta$ with $\|\eta\|_{L^\infty} \leq R \eps$. 
We establish the estimate
\begin{align} \label{inner3}
| \langle R_0(-\lambda^2) V \psi, N(\eta) \rangle | \aleq \eps^{1+\delta_1}. 
\end{align}
We deal with each term in (\ref{nonbdd}). 
Take $j=1,2$. We frequently apply (\ref{kernelbddnw}), (\ref{kernelbdd}) and H\"older:
\begin{itemize}
\item
$ \ds
\begin{aligned}[t]
|\langle R_0(-\lambda^2) V \psi, W^3 \eta^2 \rangle | 
&\aleq \|R_0(-\lambda^2) V \psi \|_{L^6} \|W^3 \eta^2 \|_{L^{6/5}} \\
&\aleq \|V \psi\|_{L^{6/5}} \|\eta\|_{L^\infty}^2 \|W^3\|_{L^{6/5}} \\
& \aleq \eps^2
\end{aligned}
$
\item
$ \ds
\begin{aligned}[t]
|\langle R_0(-\lambda^2) V \psi, \eta^5 \rangle | 
&\aleq \|R_0(-\lambda^2) V \psi \|_{L^1} \|\eta^5\|_{L^\infty} \\
&\aleq \lambda^{-2} \|V \psi\|_{L^1} \|\eta\|_{L^\infty}^5 \\
&\aleq \eps^3
\end{aligned}
$
\item
$ \ds 
\begin{aligned}[t]
\eps |\langle R_0(-\lambda^2) V \psi, \eta^{p_j} \rangle | 
&\aleq \eps \|R_0(-\lambda^2) V\psi\|_{L^1} \|\eta^{p_j}\|_{L^\infty} \\
& \aleq \eps \lambda^{-2} \|V \psi\|_{L^1} \|\eta\|_{L^\infty}^{p_j} \\
& \aleq \eps \cdot \eps^{p_j-2}
\end{aligned}
$
\end{itemize}
The term that remains requires two cases. First take $p_j>3$ then
\begin{align*}
\eps |\langle R_0(-\lambda^2) V \psi, \eta W^{p_j-1} \rangle | 
&\aleq \eps \|R_0(-\lambda^2) V \psi\|_{L^r} \|\eta W^{p_j-1}\|_{L^{r^*}} \\
& \aleq \eps \|V \psi\|_{L^q} \|\eta\|_{L^\infty} \|W^{p_j-1}\|_{L^{r^*}} \\
& \aleq \eps^2 
\end{align*}
where we have used (\ref{kernelbdd}) for some $r^*<3/2$ and $r>3$. 
Now if instead $2<p_j \leq 3$ we use (\ref{kernelbddnw}) with $r^* = (3/(p_j-1))^+$ so $1-1/r = ((p_j-1)/3)^{-}$ and
\begin{align*}
\eps |\langle R_0(-\lambda^2) V \psi, \eta W^{p_j-1} \rangle | 
&\aleq \eps \|R_0(-\lambda^2) V \psi\|_{L^r} \|\eta W^{p_j-1}\|_{L^{r^*}} \\
& \aleq \eps \lambda^{3(1-1/r)-2} \|V \psi\|_{L^1} \|\eta\|_{L^\infty} \|W^{p_j-1}\|_{L^{r^*}} \\
& \aleq \eps \cdot \eps^{(p_j-2)^-}  
\end{align*}
and so we establish (\ref{inner3}).

With the estimates (\ref{inner1}), (\ref{inner2}), (\ref{inner3}) in hand we show that a solution to (\ref{lambda}) of the desired size exists. 
For this write (\ref{lambda}) as a fixed point problem 
\begin{align}\label{Hlambda}
\lambda=\mathcal{H}(\lambda) 
:= 
\frac{\eps \langle R_0(-\lambda^2) V \psi,f(W) \rangle  
+\langle R_0(-\lambda^2) V \psi,N(\eta) \rangle} 
{\lambda \langle R_0(-\lambda^2) V \psi, W \rangle}
\end{align}
with the intention of applying Banach Fixed Point Theorem. 
We show that for a fixed $\eta$ with $\|\eta\|_{L^\infty} \aleq \eps$ the function $\mathcal{H}$ maps the interval $ \eps \lambda^{(1)}/2 \leq \lambda \leq 3 \eps \lambda^{(1)}/2$ to itself and that $\mathcal{H}$ is a contraction. 

First note that $-\langle \psi, f(W) \rangle>0$ by Assumption \ref{fassump} and so after observing 
(\ref{inner1}), (\ref{inner2}), (\ref{inner3}) we see that $\mathcal{H}(\lambda) > 0$. 
Furthermore for $\eps$ small enough we have that 
$ \eps \lambda^{(1)}/2 \leq \mathcal{H}(\lambda) \leq 3 \eps \lambda^{(1)}/2$ 
and so $\mathcal{H}$ maps this interval to itself. 

We next show that $\mathcal{H}$ is a contraction.
Take $\eps \lambda^{(1)}/2  \leq \lambda_1,\lambda_2 \leq 3 \eps \lambda^{(1)}/2$ and again keep $\eta$ fixed with $\|\eta\|_{ L^\infty} \leq R \eps$.
Write
\begin{align*}
	\mathcal{H}(\lambda)= \frac{a(\lambda)+b(\lambda)}{c(\lambda)}
\end{align*}
so that 
\begin{align*}
	|\mathcal{H}(\lambda_1)-\mathcal{H}(\lambda_2)| 
	&\leq \frac{|a_1||c_2-c_1| +|a_1-a_2||c_1| + |b_1| |c_2-c_1| +|b_1-b_2| |c_1| }{|c_1c_2|} \\
	&\aleq  |a_1-a_2| + |b_1-b_2| + \eps |c_1-c_2| 
\end{align*}
using (\ref{inner1}), (\ref{inner2}), (\ref{inner3}). 
We treat each piece in turn. 

First
\begin{align*}
	|a_1-a_2|&=
	\eps |\langle  \left( R_0(-\lambda_1^2) -R_0(-\lambda_2^2) \right) V \psi,f(W) \rangle | \\
	&= \eps |\lambda_1^2 - \lambda_2^2| | \langle R_0(-\lambda_1^2) R_0(-\lambda_2^2) V \psi, f(W) \rangle |
\end{align*}
by the resolvent identity. 
Continuing we see 
\begin{align*}
	|a_1-a_2|
	\aleq  \eps^2 |\lambda_1-\lambda_2| \|R_0(-\lambda_1^2) R_0(-\lambda_2^2) V \psi \|_{L^r} \|f(W)\|_{L^{r^*}}
\end{align*}
where $1/r + 1/r^* =1$. Note that by Assumption \ref{fassump} we have 
$f(W) \in L^{r^*}$ for some $1<r^*<3/2$ so $3<r<\infty$. 
Applying now (\ref{kernelbdd}) we get
\begin{align*}
	|a_1-a_2|\aleq  \eps^2 |\lambda_1-\lambda_2| \|R_0(-\lambda_2^2) V \psi \|_{L^q}
\end{align*}
with $3(1/q-1/r)=2$ so $1<q<3/2$. Now apply the bound (\ref{kernelbddnw}) 
\begin{align*}
	|a_1-a_2| &\aleq   \eps^2 |\lambda_1-\lambda_2| \lambda^{3(1-1/q)-2} \|V \psi\|_{L^1} \\
	& \aleq  \eps^{3(1-1/q)} |\lambda_1-\lambda_2|
\end{align*}
and note that $3(1-1/q)>0$. 

Next consider 
\begin{align*}
	|b_1-b_2|=|\langle R_0(-\lambda_1^2) V \psi,N(\eta) \rangle-\langle R_0(-\lambda_2^2) V \psi,N(\eta) \rangle|.
\end{align*}
Proceeding as in the previous argument and using (\ref{kernelbddnw}) we see
\begin{align*}
	|b_1-b_2| &\aleq \eps |\lambda_1-\lambda_2| \|R_0(-\lambda^2_1) R_0(-\lambda^2_2) V \psi \|_{L^r} \|N(\eta)\|_{L^{r^*}} \\
&\aleq \eps |\lambda_1 - \lambda_2| \lambda_1^{-2} \|R_0(-\lambda^2_2) V \psi \|_{L^r} \|N(\eta)\|_{L^{r^*}} 
\end{align*}
for $1/r + 1/r^* =1$. We can estimate this term (using different $r$ and $r^*$ for different portions of $N(\eta)$) using the computations leading to (\ref{inner3}) to achieve 
\begin{align*}
|b_1-b_1| &\aleq \eps^{-1} \cdot \eps^{1+\delta_1} |\lambda_1-\lambda_2|
= \eps^{\delta_1} |\lambda_1 - \lambda_2|.
\end{align*}

Lastly consider 
\begin{align*}
	\eps |c_1-c_2|= \eps |\lambda_1 \langle R_0(-\lambda_1^2) V \psi, W \rangle
	-\lambda_2 \langle R_0(-\lambda_2^2) V \psi, W \rangle|.
\end{align*}
Again we write $W = \sqrt{3}/|x| + \tilde{W}$ where $\tilde{W} \in L^{6/5}$. 
The second term is easy. We compute
\begin{align*}
\eps  &|\lambda_1 \langle R_0(-\lambda_1^2) V \psi, \tilde{W} \rangle
-\lambda_2 \langle R_0(-\lambda_2^2) V \psi, \tilde{W} \rangle| \\
&\aleq \eps |\lambda_1-\lambda_2| |\langle R_0(-\lambda_1^2) V \psi, \tilde{W} \rangle |
+ \eps^3 |\lambda_1-\lambda_2| | \langle R_0(-\lambda_1^2)R_0(-\lambda_2^2) V \psi, \tilde{W} \rangle | \\
& \aleq \eps |\lambda_1-\lambda_2| + \eps^3 \lambda_1^{-2} \lambda_2^{3(1-1/6)-2} |\lambda_1-\lambda_2| 
\|V \psi\|_{L^1} \|\tilde{W}\|_{L^{6/5}} \\
& \aleq \eps |\lambda_1-\lambda_2| 
\end{align*}
where we have used (\ref{Wtilde}) once and (\ref{kernelbddnw}) twice. 
For the harder term we follow the computations which establish (\ref{inner2}) and so work on the Fourier Transform side
\begin{align*}
\eps \lambda_1 & \langle R_0(-\lambda_1^2) V \psi, 1/|x| \rangle 
-\eps \lambda_2 \langle R_0(-\lambda_2^2) V \psi, 1/|x| \rangle \\
&= C\eps \lambda_1 \left\langle \frac{\hat{g}(\xi)}{|\xi|^2 + \lambda_1^2}, \frac{1}{|\xi|^2} \right\rangle 
- C\eps \lambda_2 \left\langle \frac{\hat{g}(\xi)}{|\xi|^2 + \lambda_2^2}, \frac{1}{|\xi|^2} \right\rangle \\
&= C\eps \lambda_1 \left\langle \frac{\hat{g}(\xi)-\hat{g}(0)}{|\xi|^2 + \lambda_1^2}, \frac{1}{|\xi|^2} \right\rangle 
- C\eps \lambda_2 \left\langle \frac{\hat{g}(\xi)-\hat{g}(0)}{|\xi|^2 + \lambda_2^2}, \frac{1}{|\xi|^2} \right\rangle \\
& = C\eps (\lambda_1 - \lambda_2) \left\langle \frac{\hat{g}(\xi)-\hat{g}(0)}{|\xi|^2 + \lambda_1^2}, \frac{1}{|\xi|^2} \right\rangle \\
& \quad + C \eps \lambda_2 \left\langle (\hat{g}(\xi) - \hat{g}(0)) \left( \frac{1}{|\xi|^2 + \lambda_1^2} - \frac{1}{|\xi|^2+ \lambda_2^2} \right) , \frac{1}{|\xi|^2} \right\rangle 
\end{align*} 
where we have used the fact that 
\begin{align*}
\lambda_1 \left\langle \frac{\hat{g}(0)}{|\xi|^2+\lambda_1^2} , \frac{1}{|\xi|^2} \right\rangle
= \lambda_2 \left\langle \frac{\hat{g}(0)}{|\xi|^2+\lambda_2^2} , \frac{1}{|\xi|^2} \right\rangle. 
\end{align*}
Continuing 
as in the computations used to establish (\ref{inner2}), we bound
\begin{align*}
\eps |\lambda_1 - \lambda_2|
\left| \left\langle \frac{\hat{g}(\xi)-\hat{g}(0)}{|\xi|^2 + \lambda_1^2}, \frac{1}{|\xi|^2} \right\rangle \right|
\aleq \eps  |\lambda_1-\lambda_2|
\end{align*}
and 
\begin{align*}
\eps \lambda_2 & \left| \left\langle (\hat{g}(\xi) - \hat{g}(0)) \left( \frac{1}{|\xi|^2 + \lambda_1^2} - \frac{1}{|\xi|^2+ \lambda_2^2} \right) , \frac{1}{|\xi|^2} \right\rangle \right| \\
&\aleq \eps \lambda_2 (\lambda_1+\lambda_2) |\lambda_1-\lambda_2|
\int \frac{d \xi}{|\xi|(|\xi|^2+\lambda_1^2)(|\xi|^2+\lambda_2^2)}  \\
& \aleq \eps |\lambda_1-\lambda_2| \int \frac{d \zeta}{|\zeta|(|\zeta|^2 +1)(|\zeta|^2 + \lambda_2^2/\lambda_1^2)} \\
& \aleq \eps |\lambda_1-\lambda_2|.
\end{align*}
In this way we finally have
\begin{align*}
\eps |c_1-c_2| \leq \eps |\lambda_1-\lambda_2|. 
\end{align*}

So, putting everything together we see that by taking $\eps$ sufficiently small,
\begin{align*}
	|\mathcal{H}(\lambda_1)-\mathcal{H}(\lambda_2)| < \kappa |\lambda_1-\lambda_2|
\end{align*}
for some $0 < \kappa < 1$,
and hence $\mathcal{H}$ is a contraction. 
Therefore (\ref{Hlambda}) has a unique fixed point of the desired size.

To find the leading order $\lambda^{(1)}$ let $\lambda$ take the form in (\ref{lambdaexp}), substitute to (\ref{lambda}) use estimates (\ref{inner1}), (\ref{inner2}), (\ref{inner3}) and ignore higher order terms. 
An inspection of the higher order terms gives the order of $\tilde{\lambda}$. 
\end{proof}

In this way we now think of $\lambda$ as a function of $\eta$. 
We will also need the following Lipshitz condition for what follows in Lemma \ref{etalemma}. 
\begin{lemma}\label{lipshitz}
The $\lambda$ generated via Lemma \ref{lambdalemma} is Lipshitz continuous in $\eta$ in the sense that
\begin{align*}
	|\lambda_1-\lambda_2| \aleq  \eps^{\delta_1} \|\eta_1-\eta_2\|_{L^{\infty} }.
\end{align*}
\end{lemma}
\begin{proof}
Take $\eta_1$ and $\eta_2$ with $\|\eta_1\|_{L^\infty},\|\eta_2\|_{L^\infty} \leq R \eps$. 
Let $\eta_1$ and $\eta_2$ give rise to $\lambda_1$ and $\lambda_2$ respectively through Lemma \ref{lambdalemma}. 
Consider now the difference 
\begin{align*}
	|\lambda_1-\lambda_2| 
	=& \bigg| \frac{ \eps \langle R_0(-\lambda_1^2) V \psi,f(W) \rangle 
	+\langle R_0(-\lambda_1^2) V \psi,N(\eta_1) \rangle}
	{\lambda_1 \langle R_0(-\lambda_1^2) V \psi, W \rangle} \\
	& \quad \quad -\frac{ \eps \langle R_0(-\lambda_2^2) V \psi,f(W) \rangle 
	+\langle R_0(-\lambda_2^2) V \psi,N(\eta_2) \rangle}
	{\lambda_2 \langle R_0(-\lambda_2^2) V \psi, W \rangle} \bigg| \\
	=: & \left| \frac{a(\lambda_1) + b(\lambda_1,\eta_1)}{c(\lambda_1)} 
	- \frac{a(\lambda_2)+b(\lambda_2,\eta_2)}{c(\lambda_2)} \right| \\
\end{align*}
observing (\ref{Hlambda}). Now we estimate
\begin{align*}
	|\lambda_1-\lambda_2| 
	\leq & \left| \frac{b(\lambda_1,\eta_1)-b(\lambda_1,\eta_2)}{c(\lambda_1)} \right|
	+ \left| \frac{a(\lambda_1) + b(\lambda_1,\eta_2)}{c(\lambda_1)}
	- \frac{a(\lambda_2) + b(\lambda_2,\eta_2)}{c(\lambda_2)} \right| \\
	\leq&  C |\langle R_0(-\lambda_1^2) V \psi, N(\eta_1) - N(\eta_2) \rangle | + \kappa|\lambda_1-\lambda_2|
\end{align*}
for some $0<\kappa<1$. 
The second term has been estimated using the computations of Lemma \ref{lambdalemma} and taking $\eps$ small enough. 
Now we estimate the first. 
Observing the terms in (\ref{nondiff}) we use the same procedure that established (\ref{inner3}) to obtain 
\begin{align*} 
|\langle R_0&(-\lambda_1^2) V \psi, N(\eta_1) - N(\eta_2) \rangle | 
\aleq  \eps^{\delta_1} \|\eta_1 - \eta_2\|_{L^{\infty}}. 
\end{align*}
So together we now see
\begin{align*}
	(1-\kappa)|\lambda_1-\lambda_2| \aleq  \eps^{\delta_1} \|\eta_1-\eta_2\|_{L^{\infty}}
\end{align*}
which gives the desired result. 
\end{proof}

\subsection{Solving for the Correction
} \label{findingeta}

We next solve (\ref{eta}), given that (\ref{orth}) holds. 
Recall the formulation of~\eqref{eta} as the fixed-point equation
\[
  \eta = \mathcal{G}(\eta,\eps) = (H + \lambda^2)^{-1} \mathcal{F}
\]
where in light of Lemma~\ref{lambdalemma}, we take
$\lambda = \lambda(\eps, \eta)$ and
$\mathcal{F} = \mathcal{F}(\eps, \lambda(\eps,\eta), \eta)$
so that~\eqref{orth} holds.

\begin{lemma} \label{etalemma}
There exists $R_0>0$ such that for any $R \geq R_0$, there is 
$\eps_1 = \eps_1(R) > 0$ such that for each $0<\eps \leq \eps_1$, there exists a 
unique solution $\eta \in L^{\infty}$ to \emph{(\ref{eta})} 
with $\|\eta\|_{L^\infty} \leq R \eps$. 
Moreover, we have the expansion 
\begin{align*}
\eta = \eps \bar{Q}(1+G_0V)^{-1}\bar{P} \left( G_0 f(W) - \lambda^{(1)} \sqrt{3} \lambda R_0(-\lambda^2) |x|^{-1} \right)
+ O_{L^\infty}(\eps^{1+\delta_1}) 
\end{align*}
where $\bar{P}$ and $\bar{Q}$ are given in \eqref{projections}. 
\end{lemma}

\begin{proof}
We proceed by means of Banach Fixed Point Theorem. We show that 
$\mathcal{G}(\eta)$ maps a ball to itself and is a contraction. 
In this way we establish a solution to $\eta = \mathcal{G}(\eta,\eps,\lambda(\eps,\eta))$ in (\ref{eta}). 

Let $R>0$ (to be chosen) and take $\eps<\eps_0(R)$ as in Lemma \ref{lambdalemma}. 
In this way given $\eta \in L^\infty$ with $\|\eta\|_{L^\infty} \leq R \eps$ we can generate
\[
  \lambda = \lambda(\eps,\eta) = \lambda^{(1)} \eps + o(\eps).
\] 
We aim to take $\eps$ smaller still in order to run fixed point in the $L^\infty$ ball of radius $R \eps$. 

Consider
\begin{align*}
\|\mathcal{G}\|_{L^{\infty}} &= \| (1+R_0(-\lambda^2)V)^{-1} R_0(-\lambda^2) \mathcal{F}\|_{L^{\infty}} \\
&\aleq  \|R_0(-\lambda^2) \mathcal{F}\|_{L^{\infty}}
\end{align*}
in light of Lemma \ref{fullresolvent} and since we have chosen $\lambda$ to satisfy (\ref{orth}). Continuing with
\begin{align*}
\|\mathcal{G}\|_{L^{\infty}} \aleq  \|R_0(-\lambda^2) \left( -\lambda^2 W + \eps f(W) + N(\eta) \right) \|_{L^{\infty}}
\end{align*}
we treat each term separately. 
For the first term it is sufficient to replace $W$ with $1/|x|$ (otherwise we simply apply (\ref{kernelinfty}))
\begin{align*}
\lambda^2 \|R_0(-\lambda^2)W\|_{L^\infty} 
& \aleq \lambda \left\| \lambda R_0(-\lambda^2) \frac{1}{|x|} \right\|_{L^\infty} \\
& \aleq  \lambda \left\| \lambda \int \frac{e^{-\lambda |y|}}{|y|} \frac{1}{|x-y|} dy \right\|_{L^\infty} \\
& \aleq  \lambda \left\| \int \frac{ e^{-|z|}}{|z|} \frac{1}{|\lambda x - z|} dz \right\|_{L^\infty} \\
& \aleq  \lambda \left\| \left( \frac{ e^{-|x|}}{|x|} * \frac{1}{|x|} \right) (\lambda x) \right\|_{L^\infty} \\
& \aleq  \lambda \aleq \eps. 
\end{align*}
Now for the second term use (\ref{kernelinfty}) 
\begin{align*}
\eps \|R_0(-\lambda^2)f(W)\|_{L^\infty} \aleq \eps \|f(W)\|_{L^{3/2^-} \cap L^{3/2^+}} \aleq \eps. 
\end{align*}
And for the higher order terms we employ (\ref{kernelbddnw}) and (\ref{kernelinfty})
\begin{itemize}
\item
$ \ds
\begin{aligned}[t]
\|R_0(-\lambda^2) (W^3 \eta^2) \|_{L^{\infty}} 
&\aleq \| W^3 \eta^2 \|_{L^{3/2^-} \cap L^{3/2^+} } \\
&\aleq \|W^3 \|_{L^{3/2^-} \cap L^{3/2^+}}  \| \eta\|_{L^{\infty}}^2 \\
&\aleq R^2 \eps^{2} 
\end{aligned}
$
\item
$
\|R_0(-\lambda^2) \eta^5 \|_{L^{\infty}} \aleq \lambda^{-2} \|\eta^5\|_{L^{\infty}} 
\aleq \lambda^{-2} \|\eta\|_{L^{\infty}}^5 \aleq R^5 \eps^{3}
$
\item
$ \ds
\begin{aligned}[t]
\eps \|R_0(-\lambda^2) \eta^{p_j} \|_{L^{\infty}} 
&\aleq \eps \lambda^{-2} \|\eta^{p_j} \|_{L^{\infty}} \\
&\aleq \eps^{-1} \|\eta\|_{L^{\infty}}^{p_j} \\
&\aleq R^{p_j} \eps^{p_j -1}
\end{aligned}
$
\end{itemize}
for $j=1,2$. 
The remaining remainder term again requires two cases. For $p_j>3$ we use (\ref{kernelinfty}) to see
\begin{align*}
\eps \left\|R_0(-\lambda^2) \left( \eta W^{p_j-1} \right) \right\|_{L^{\infty}} 
&\aleq \eps \|\eta W^{p_j-1} \|_{L^{3/2^-} \cap L^{3/2^+}} \\
&\aleq \eps \|\eta\|_{L^\infty} \|W^{p_j-1} \|_{L^{3/2^-} \cap L^{3/2^+}} \\
& \aleq R \eps^2
\end{align*}
and for $2<p_j \leq 3$ we apply (\ref{kernelbddnw})
\begin{align*}
\eps \left\|R_0(-\lambda^2) \left( \eta W^{p_j-1} \right) \right\|_{L^{\infty}} 
&\aleq \eps \lambda^{(p_j-1)^- -2} \|\eta\|_{L^\infty} \|W^{p_j-1} \|_{L^{3/(p_j-1)^+}} \\
&\aleq  R \eps^{1+(p_j-2)^-}
\end{align*}
Collecting the above yields
\begin{align} \label{Ginfty}
\|\mathcal{G}\|_{L^\infty} \leq C \eps \left( 1 +  R^2 \eps + R^5 \eps^2 
+ R^{p_1} \eps^{p_1 - 2} + R^{p_2} \eps^{p_2 - 2} + R \eps
+ R \eps^{(p_1-2)-} \right)
\end{align}
and so taking $R_0 = 2C$, $R \geq R_0$, and then 
$\eps$ small enough so that 
$R \eps + R^4 \eps^2 
+ R^{p_1-1} \eps^{p_1 - 2} + R^{p_2-1} \eps^{p_2 - 2} + \eps
+ \eps^{(p_1-2)-} \leq \frac{1}{2C}$, we arrive at
\[
  \|\mathcal{G}\|_{L^\infty} \leq  R \eps.
\]
Hence $\mathcal{G}$ maps the ball of radius $R \eps$ in $L^\infty$ to itself.

Now we show that $\mathcal{G}$ is a contraction. 
Take $\eta_1$ and $\eta_2$ and let them give rise to $\lambda_1$ and $\lambda_2$ respectively. 
Again $\|\eta_j\|_{L^{\infty}} \leq R \eps$  
and denote $\mathcal{F}(\eta_j)$ by $\mathcal{F}_j$, $j=1,2$.
Consider
\begin{align*}
&\|\mathcal{G}(\eta_1,\eps) - \mathcal{G}(\eta_2,\eps)\|_{L^{\infty}} \\
&= \|(1+R_0(-\lambda^2_1)V)^{-1} R_0(-\lambda^2_1) \mathcal{F}_1
-(1+R_0(-\lambda^2_2)V)^{-1} R_0(-\lambda^2_2) \mathcal{F}_2 \|_{L^{\infty}} \\
& \leq \|(1+R_0(-\lambda^2_1)V)^{-1} \left( R_0(-\lambda^2_1) \mathcal{F}_1 - R_0(-\lambda^2_2) \mathcal{F}_2
\right)\|_{L^{\infty}} \\
& \quad + \|\left( (1+R_0(-\lambda^2_1)V)^{-1} -(1+R_0(-\lambda^2_2)V)^{-1} \right) R_0(-\lambda^2_2) \mathcal{F}_2\|_{L^{\infty}} \\
& \leq \| R_0(-\lambda^2_1) \mathcal{F}_1 - R_0(-\lambda^2_2) \mathcal{F}_2 \|_{L^{\infty}} \\
& \quad + \|\left( (1+R_0(-\lambda^2_1)V)^{-1} -(1+R_0(-\lambda^2_2)V)^{-1} \right) R_0(-\lambda^2_2) \mathcal{F}_2\|_{L^{\infty}} \\
& \leq \| R_0(-\lambda^2_1)\left( \mathcal{F}_1 - \mathcal{F}_2 \right) \|_{L^{\infty}}
 + \| \left( R_0(-\lambda_1^2) - R_0(-\lambda^2_2) \right) \mathcal{F}_2 \|_{L^{\infty}} \\
& \quad + \|\left( (1+R_0(-\lambda^2_1)V)^{-1} -(1+R_0(-\lambda^2_2)V)^{-1} \right) R_0(-\lambda^2_2) \mathcal{F}_2\|_{L^{\infty}} \\
&=: \text{I} + \text{II} + \text{III} 
\end{align*}
where we have applied Lemma \ref{fullresolvent}, observing the orthogonality condition. 
We treat each part in turn. 

Start with \textrm{I}. This computation is similar to those previous. We also apply Lemma \ref{lipshitz}:
\begin{align*}
\| R_0(-\lambda^2_1)\left( \mathcal{F}_1 - \mathcal{F}_2 \right) \|_{L^{\infty}}
&= \|R_0(-\lambda^2_1) \left( (\lambda_2^2 - \lambda_1^2) W + N(\eta_1)-N(\eta_2) \right) \|_{L^{\infty}} \\
& \aleq  |\lambda_1 - \lambda_2|  +  \eps^{\delta_1} \|\eta_1-\eta_2\|_{L^{\infty}} \\
&\aleq \eps^{\delta_1} \|\eta_1-\eta_2\|_{L^{\infty}}. 
\end{align*}

Part \textrm{II} is also similar to  previous computations:
\begin{align*}
\| \left( R_0(-\lambda^2_1) - R_0(-\lambda^2_2) \right) \mathcal{F}_2 \|_{L^{\infty}} 
&= |\lambda_1^2 - \lambda_2^2| \|R_0(-\lambda^2_1)R_0(-\lambda^2_2) \mathcal{F}_2 \|_{L^{\infty}} \\
& \aleq |\lambda_1 + \lambda_2| |\lambda_1 -\lambda_2| \lambda_1^{-2} \|R_0(-\lambda^2_2) \mathcal{F}_2 \|_{L^{\infty}} \\
& \aleq  \eps^{1} \cdot \eps^{-2} \cdot \eps |\lambda_1-\lambda_2|  \\
& \aleq  |\lambda_1-\lambda_2| \\
& \aleq \eps^{\delta_1} \|\eta_1-\eta_2\|_{L^{\infty}}. 
\end{align*}

Part \textrm{III} is the hardest. First we find a common denominator 
\begin{align*}
  (1+R_0(&-\lambda^2_1)V)^{-1} -(1+R_0(-\lambda^2_2)V)^{-1}  \\
&= (1+R_0(-\lambda^2_1)V)^{-1} (1+R_0(-\lambda^2_2)V) (1+R_0(-\lambda^2_2)V)^{-1} \\
&\quad - (1+R_0(-\lambda^2_1)V)^{-1} (1+R_0(-\lambda^2_1)V) (1+R_0(-\lambda^2_2)V)^{-1} \\
&= (1+R_0(-\lambda^2_1)V)^{-1} \left(  R_0(-\lambda^2_2)V - R_0(-\lambda^2_1)V \right)  (1+R_0(-\lambda^2_2)V)^{-1}
\end{align*}
so that
\begin{align*}
&\left((1+R_0(-\lambda^2_1)V)^{-1} -(1+R_0(-\lambda^2_2)V)^{-1} \right) R_0(-\lambda^2_2) \mathcal{F}_2 = \\
&(1+R_0(-\lambda^2_1)V)^{-1} \left(  R_0(-\lambda^2_2)V - R_0(-\lambda^2_1)V \right)  (1+R_0(-\lambda^2_2)V)^{-1}
 R_0(-\lambda^2_2) \mathcal{F}_2 \\
&= (1+R_0(-\lambda^2_1)V)^{-1} \left(  R_0(-\lambda^2_2)V - R_0(-\lambda^2_1)V \right) \mathcal{G}(\eta_2).
\end{align*}
Now
\begin{align*}
\text{III} = \|
(1+R_0(-\lambda^2_1)V)^{-1} \left(  R_0(-\lambda^2_2)V - R_0(-\lambda^2_1)V \right) \mathcal{G}(\eta_2) \|_{L^{\infty}}
\end{align*}
and here we just suffer the loss of one $\lambda$ (Lemma \ref{fullresolvent}) to achieve
\begin{align*}
\text{III} &\aleq
 \lambda_1^{-1} \| 
\left(  R_0(-\lambda^2_2)V - R_0(-\lambda^2_1)V \right) \mathcal{G}(\eta_2) \|_{L^{\infty}} \\
& \aleq  \lambda_1^{-1} |\lambda^2_2 - \lambda^2_1| \| R_0(-\lambda^2_2) R_0(-\lambda^2_1) V 
\mathcal{G}(\eta_2)\|_{L^{\infty}} \\
& \aleq  \lambda_1^{-1} |\lambda_2+\lambda_1| |\lambda_2 - \lambda_1| \lambda_2^{-1/2} \|R_0(-\lambda_1^2) V \mathcal{G}(\eta_2)\|_{L^{2}} \\
& \aleq  \eps^{-1} \cdot \eps^1 |\lambda_2 - \lambda_1| \lambda_2^{-1/2}\lambda_1^{-1/2} \|V \mathcal{G}(\eta_2)\|_{L^1} \\
& \aleq  \eps^{-1} |\lambda_2-\lambda_1| \|V\|_{L^1} \|\mathcal{G}(\eta_2)\|_{L^\infty}
\end{align*}
and using Lemma \ref{lipshitz} and (\ref{Ginfty}) we see
\begin{align*}
\text{III} \aleq |\lambda_1-\lambda_2|
\aleq \eps^{\delta_1} \|\eta_1-\eta_2\|_{L^{\infty}}. 
\end{align*}
Hence, by taking $\eps$ smaller still if needed, we have
\begin{align*}
\|\mathcal{G}(\eta_1,\eps) - \mathcal{G}(\eta_2,\eps)\|_{L^{\infty}} \leq \kappa \|\eta_1 - \eta_2\|_{L^{\infty}}
\end{align*}
for some $0<\kappa<1$ and so $\mathcal{G}$ is a contraction. 
Therefore, invoking the Banach fixed-point theorem,
we have established the existence of a unique $\eta$, 
with $\| \eta \|_{L^\infty} \leq R \eps$, satisfying (\ref{eta}).

To see the leading order observe the order of the terms appearing in the previous computations as well as the following. First if $3\leq p_1<5$ then
\begin{align*}
\eps \|\left(R_0(-\lambda^2) - G_0 \right) f(W) \|_{L^\infty}
&\aleq \eps \lambda^{2} \|R_0(-\lambda^2) G_0 f(W)\|_{L^\infty} \\
&\aleq \eps \lambda^2 \cdot \lambda^{-1^-} \|G_0 f(W) \|_{L^{3^+}} \\
& \aleq \eps \lambda^{1^-} \|f(W)\|_{L^{1^+}} \\
& \aleq \eps^{2^-}
\end{align*}
and if instead $2<p_1<3$ then take $3/q=(p_1-2)^-$ and 
\begin{align*}
\eps \|\left(R_0(-\lambda^2) - G_0 \right) f(W) \|_{L^\infty}
&\aleq \eps \lambda^{2} \|R_0(-\lambda^2) G_0 f(W)\|_{L^\infty} \\
& \aleq \eps \lambda^2 \cdot \lambda^{3/q-2} \|G_0 f(W)\|_{L^q} \\
& \aleq \eps \lambda^{3/q} \|f(W)\|_{L^{(3/p_1)^+}} \\
& \aleq \eps^{1 + (p_1-2)^-}.
\end{align*}
The lemma is now proved. 
\end{proof}

With the existence of $\eta$ established we can improve the space in which $\eta$ lives.

\begin{lemma} \label{etabound}
The $\eta$ established in Lemma \ref{etalemma} is in $L^{r} \cap \dot{H}^1$ for any $3<r \leq \infty$. 
The function $\eta$ also enjoys the bounds
\begin{alignat*}{2}
&\|\eta\|_{L^r} &&\aleq  \eps^{1-3/r} \\
&\|\eta\|_{\dot{H}^1} &&\aleq  \eps^{1/2} 
\end{alignat*}
for all $3<r \leq \infty$.
Furthermore we have the expansion 
\begin{align*}
\eta = \bar{Q}(1+G_0V)^{-1}\bar{P} R_0(-\lambda^2) (-\lambda^2 \sqrt{3} |x|^{-1}) + \tilde{\eta}
\end{align*}
with 
\begin{align*}
&\|\tilde{\eta}\|_{L^r}
\lesssim \\
&\max
{\left\lbrace
{\left\lbrace
\begin{array}{c}
\eps^{1-}, \emph{ if } 2<p_1<3 \emph{ and } r=3/(p_1-2)  \\
\eps, \quad \emph{else}
\end{array}
\right\rbrace}
, 
\eps^{p_1-2  +1-3/r}
, 
\eps^{2(1-3/r)}
\right\rbrace}
\end{align*}
for $3<r<\infty$ and 
where $\bar{P}$ and $\bar{Q}$ are given in \emph{(\ref{projections})}. 
\end{lemma}

\begin{proof}
The computations which produce (\ref{Ginfty}) are sufficient to establish the result with $r=\infty$. 
Take $3<r<\infty$ and consider:
\begin{align*}
\|\eta\|_{L^r} \aleq \lambda^2 \|R_0(-\lambda^2)W\|_{L^r} +  \eps \|R_0(-\lambda^2)f(W)\|_{L^r} 
+ \|R_0(-\lambda^2) N(\eta)\|_{L^r}.
\end{align*}

For the first term use (\ref{kernelbddw})
\begin{align*}
	\lambda^2 \|R_0(-\lambda^2)W\|_{L^r} \aleq \lambda^2 \cdot \lambda^{3(1/3-1/r)-2} \|W\|_{L^3_w}
	\aleq \eps^{1-3/r} 
\end{align*}
to see the leading order contribution. 

While the second term contributed to the leading order in Lemma \ref{etalemma} it is inferior to the first term when measured in $L^r$. We do however need several cases. 
Suppose that $3 \leq p_1 <5$ or $r> 3/(p_1-2)$ and apply (\ref{kernelbdd}) with $1/q=1/r+2/3$
\begin{align*}
\eps \|R_0(-\lambda^2)f(W)\|_{L^r} 
\aleq \eps \|f(W)\|_{L^q} 
\aleq \eps. 
\end{align*}
Note that under these conditions $f(W) \in L^q$. 
Now suppose that $2<p_1<3$ and $r=3/(p_1-2)$ and apply (\ref{kernelbddnw}) with $q=(3/p_1)^+$
\begin{align*}
\eps \|R_0(-\lambda^2)f(W)\|_{L^r} 
\aleq \eps \lambda^{3(1/q-(p_1-2)/3)-2} \|f(W)\|_{L^q}
\aleq \eps^{1^-}. 
\end{align*}
And if $2<p_1<3$ and $3<r<3/(p_1-2)$ apply (\ref{kernelbddw}) with $q=3/p_1$
\begin{align*}
\eps \|R_0(-\lambda^2)f(W)\|_{L^r} 
\aleq \eps \lambda^{3(p_1/3-1/r)-2} \|f(W)\|_{L^q_w}
\aleq \eps^{1-3/r+p_1-2}. 
\end{align*}

And thirdly the remaining terms.
First use (\ref{kernelbdd}) where $1/q = 1/r + 2/3$ to see 
\begin{align*}
	\|R_0(-\lambda^2) (W^3 \eta^2)\|_{L^r} \aleq \|W^3 \eta^2\|_{L^q}
	\aleq \|W^3\|_{L^{3/2}} \|\eta\|_{L^r} \|\eta\|_{L^\infty}
	\aleq \eps \|\eta\|_{L^r} 
\end{align*}
and now use (\ref{kernelbddnw}) with $1/q=1/r$ to obtain
\begin{align*}
\|R_0(-\lambda^2) \eta^5\|_{L^r} \aleq \lambda^{-2} \|\eta^5\|_{L^r} 
\aleq \lambda^{-2} \|\eta\|_{L^r} \|\eta\|_{L^{\infty}}^4 
\aleq \eps^2 \|\eta\|_{L^r}.
\end{align*}
and similarly for $j=1,2$ 
\begin{align*}
\eps \|R_0(-\lambda^2) \eta^{p_j} \|_{L^r} 
\aleq \eps \lambda^{-2} \|\eta^{p_j}\|_{L^r}
\aleq \eps^{-1} \|\eta\|_{L^\infty}^{p_j-1} \|\eta\|_{L^r}
\aleq \eps^{p_j-2} \|\eta\|_{L^r}
\end{align*}
noting that $p_2-2 \geq p_1-2>0$. 
For the last remainder term we have two cases.
If $p_j>3$ then use (\ref{kernelbdd}) with $1/q = 1/r+2/3$
\begin{align*}
\eps \|R_0(-\lambda^2) \left( \eta W^{p_j-1}  \right) \|_{L^r} 
\aleq \eps \|\eta W^{p_j-1}\|_{L^q} 
\aleq \eps \|\eta\|_{L^r} \| W^{p_j-1}\|_{L^{3/2}}
\aleq \eps \|\eta\|_{L^r}. 
\end{align*}
If instead $2<p_j \leq 3$ then we need (\ref{kernelbddw}) with $1/q = (p_j-1)/3$ so that
\begin{align*}
\eps \|R_0(-\lambda^2) \left( \eta W^{p_j-1}  \right) \|_{L^r} 
&\aleq \eps \lambda^{3(1/q-1/r)-2} \|\eta W^{p_j-1} \|_{L^q_w} \\
&\aleq \eps \lambda^{p_1-1-3/r-2} \|\eta\|_{L^\infty} \|W^{p_j-1}\|_{L^q_w} \\
& \aleq \eps^{p_j-2} \eps^{1-3/r} 
\end{align*}
So together we have
\begin{align*}
  \|\eta\|_{L^r} \leq C \eps^{1-3/r} + \kappa \| \eta \|_{L^r}
\end{align*}
where $\kappa$ may be chosen sufficiently small to yield the
desired $L^r$ bound for $3<r<\infty$. 
An inspection of the higher order terms gives the size of $\tilde{\eta}$.
We also must note Lemma \ref{fullresolvent}. 
There are several competing terms which determine the size of $\tilde{\eta}$ depending on $p_1$ and $r$. 

On to the $\dot{H}^1$ norm. We need the identity 
\begin{align*}
\eta &= (1+ R_0(-\lambda^2)V)^{-1} R_0(-\lambda^2)\mathcal{F} \\
&= R_0(-\lambda^2)\mathcal{F} 
- R_0(-\lambda^2) V (1 + R_0(-\lambda^2)V)^{-1} R_0(-\lambda^2) \mathcal{F} \\
&= R_0(-\lambda^2)\mathcal{F} - R_0(-\lambda^2) V \eta
\end{align*}
so we have two parts
\begin{align*}
	\|\eta\|_{\dot{H}^1} \leq \|R_0(-\lambda^2)\mathcal{F}\|_{\dot{H}^1} + \|R_0(-\lambda^2) V \eta \|_{\dot{H}^1}.
\end{align*}
For the first
\begin{align*}
\|R_0(-\lambda^2) \mathcal{F}\|_{\dot{H}^1} 
&\aleq \lambda^2 \|R_0(-\lambda^2)W\|_{\dot{H}^1} + \eps \|R_0(-\lambda^2) f(W+\eta) \|_{\dot{H}^1} \\
& \quad + \|R_0(-\lambda^2) \left( W^3 \eta^2 + W^2 \eta^3 + W \eta^4 + \eta^5 \right) \|_{\dot{H}^1} \\
\end{align*}
and 
\begin{align*}
\lambda^2 \|R_0(-\lambda^2)W\|_{\dot{H}^1}
&\aleq \lambda^2 \|R_0(-\lambda^2) \nabla W \|_{L^2} \\
& \aleq \lambda^2 \cdot \lambda^{1/2-2} \|\nabla W \|_{L^{3/2}_w} \\
& \aleq \eps^{1/2}
\end{align*}
and
\begin{align*}
\eps \|R_0(-\lambda^2) f(W+\eta) \|_{\dot{H}^1}
&\aleq \eps \|R_0(-\lambda^2) f'(W+\eta) (\nabla W + \nabla \eta) \|_{L^2} \\
&\aleq  \eps \lambda^{-1/2} \|f'(W+\eta) \nabla W \|_{L^1} \\
& \quad + \eps \lambda^{1^-} \|f'(W+\eta) \nabla \eta\|_{L^{6/5^-}} \\
&\aleq \eps^{1/2} \|f'(W+\eta)\|_{L^{3^-}} \|\nabla W\|_{L^{3/2^+}} \\
& \quad + \eps^{0^+} \|f'(W+\eta)\|_{L^{3^-}} \|\nabla \eta\|_{L^2} \\
& \aleq \eps^{1/2} + \kappa \|\eta\|_{\dot{H}^1}
\end{align*}
with $\kappa$ small and  
\begin{align*}
\|R_0(-\lambda^2) &\left( W^3 \eta^2 + W^2 \eta^3  + W \eta^4 + \eta^5 \right) \|_{\dot{H}^1} 
\aleq \|R_0(-\lambda^2) \eta (\nabla W f_1 + \nabla \eta f_2) \|_{L^2} 
\end{align*}
where $f_1$ and $f_2$ are in $L^2$ so 
\begin{align*}
\|R_0(-\lambda^2) &\left( W^3 \eta^2 + W^2 \eta^3  + W \eta^4 + \eta^5 \right) \|_{\dot{H}^1} \\
&\aleq \lambda^{-1/2} \|\eta\|_{L^\infty}  \left( \|\nabla W\|_{L^2} \|f_1\|_{L^2} + \|\nabla \eta\|_{L^2} \|f_2\|_{L^2}
\right)\\
& \aleq \eps^{1/2} + \kappa \|\eta\|_{\dot{H}^1}. 
\end{align*}
For the second
\begin{align*}
\|R_0(-\lambda^2)V\eta \|_{\dot{H}^1} 
= \left\|\left( \nabla \frac{e^{-\lambda|x|}}{|x|}\right)* \left( V \eta \right)\right\|_{L^2}
= \left\|\left( \lambda^2 g(\lambda x) \right) * \left( V \eta \right)\right\|_{L^2}
\end{align*}
where $g \in L^{3/2}_w$. So using weak Young's we obtain
\begin{align*}
\|R_0(-\lambda^2)V\eta \|_{\dot{H}^1} 
&\aleq \lambda^2 \|g(\lambda x)\|_{L^{3/2}_w} \|V \eta \|_{L^{6/5}} \\
&\aleq \lambda^2 \cdot \lambda^{-2} \|V\|_{L^{3/2}} \|\eta\|_{L^6} \\
& \aleq \|\eta\|_{L^6} \\
& \aleq \eps^{1/2}. 
\end{align*}
So putting everything together gives
\begin{align*}
	\|\eta\|_{\dot{H}^1} \leq C \left( \eps^{1/2} + \kappa \| \eta \|_{\dot H^1} \right) 
\end{align*}
which gives the desired bound by taking $\kappa$ sufficiently small. 
\end{proof}

Combining Lemmas \ref{lambdalemma}, \ref{etalemma}, \ref{etabound} and Remark \ref{H1dot} completes the proof of Theorem \ref{theorem1}. 

At this point we demonstrate the following monotonicity result which will be used in Section \ref{variational}. 
\begin{lemma}
\label{monotone}
Suppose that $f(W)=W^p$ with $3<p<5$. 
Take $\eps_1$ and $\eps_2$ with $0<\eps_1<\eps_2<\eps_0$. 
Let $\eps_1$ give rise to $\lambda_1$ and $\eta_1$ and let $\eps_2$ give rise to $\lambda_2$ and $\eta_2$ via Theorem \ref{theorem1}. 
We have
\begin{align}
|(\lambda_2 - \lambda_1) - \lambda^{(1)}(\eps_2-\eps_1) | \lesssim o(1) |\eps_2-\eps_1|. 
\end{align}
\end{lemma}
\begin{proof}
We first establish the estimate
\begin{align}\label{monostep}
|\lambda_2-\lambda_1| \leq \left( \lambda^{(1)} + o(1) \right) |\eps_2-\eps_1|
\end{align}
We write, as in Lemma \ref{lambdalemma} and Lemma \ref{lipshitz}
\begin{align*}
\lambda_2 - \lambda_1 &= \frac{ a(\eps_2,\lambda_2) + b(\eps_2,\lambda_2,\eta_2)}{c(\lambda_2)} 
- \frac{ a(\eps_1,\lambda_1) + b(\eps_1,\lambda_1,\eta_1)}{c(\lambda_1)} \\
&=\frac{a(\eps_2,\lambda_2) - a(\eps_1,\lambda_2) + b(\eps_2,\lambda_2,\eta_2) - b(\eps_1,\lambda_2,\eta_2)}{c(\lambda_2)}
\\ &\quad + \frac{ a(\eps_1,\lambda_2) + b(\eps_1,\lambda_2,\eta_2)}{c(\lambda_2)}
- \frac{ a(\eps_1,\lambda_1) + b(\eps_1,\lambda_1,\eta_1)}{c(\lambda_1)}. 
\end{align*}
The second line, containing only $\eps_1$ and not $\eps_2$, has been dealt with in the proof of Lemma \ref{lipshitz} and so there follows 
\begin{align*}
|\lambda_2 - \lambda_1| &\leq \left|
\frac{a(\eps_2,\lambda_2) - a(\eps_1,\lambda_2) + b(\eps_2,\lambda_2,\eta_2) - b(\eps_1,\lambda_2,\eta_2)}{c(\lambda_2)}
\right|
\\ & \quad 
+ o(1) \|\eta_2-\eta_1\|_{L^\infty} + o(1)|\lambda_2-\lambda_1| \\
& \leq |\eps_2-\eps_1| \left( \lambda^{(1)} + o(1) \right) + o(1)\|\eta_2-\eta_1\|_{L^\infty} + o(1)|\lambda_2-\lambda_1|
.
\end{align*}
For the $\eta$'s we estimate
\begin{align*}
\|\eta_2-\eta_1\|_{L^\infty} \lesssim o(1)\|\eta_2-\eta_1\|_{L^\infty} + |\lambda_2-\lambda_1| + |\eps_2-\eps_1|
\end{align*}
appealing to Lemma \ref{etalemma}.
So putting everything together we have 
\begin{align*}
|\lambda_2-\lambda_1| \leq \left( \lambda^{(1)} + o(1) \right) |\eps_2-\eps_1|  
\end{align*}
establishing (\ref{monostep}). 

Now we proceed to the more refined (\ref{monotone}). 
Observing the computations leading to (\ref{monostep}) we have 
\begin{align*}
|\lambda_2 - \lambda_1 - (\eps_2-\eps_1)\lambda^{(1)}|
\leq \left| \frac{a(\eps_2,\lambda_2) - a(\eps_1,\lambda_2)}{c(\lambda_2)} - (\eps_2-\eps_1)\lambda^{(1)}\right| \\
+o(1)\|\eta_2-\eta_1\|_{L^\infty} +o(1)|\lambda_2-\lambda_1|.  
\end{align*}
By (\ref{monostep}) the last two terms are of the correct size and so we focus on the first. 
We have
\begin{align*}
&\left| \frac{a(\eps_2,\lambda_2) - a(\eps_1,\lambda_2)}{c(\lambda_2)} - (\eps_2-\eps_1)\lambda^{(1)}\right| \\
& \quad \quad \quad \quad \quad \quad \quad \quad \quad \quad \quad \quad
=\left| (\eps_2-\eps_1) \left( \frac{\langle R_0(-\lambda_2^2) V \psi, W^p \rangle}
{\lambda_2 \langle R_0(-\lambda_2^2) V \psi,W \rangle} - \lambda^{(1)} \right) \right| \\
& \quad \quad \quad \quad \quad \quad \quad \quad \quad \quad \quad \quad
= o(1) |\eps_2-\eps_1|
\end{align*}
noting (\ref{inner1}) and (\ref{inner2}). 
And so, putting everything together we achieve 
\begin{align*}
|\lambda_2 - \lambda_1 - (\eps_2-\eps_1)\lambda^{(1)}|
\lesssim o(1)|\eps_2-\eps_1|
\end{align*}
as desired. 
\end{proof}

\section{
Variational Characterization
} \label{variational}

It is not clear from the construction that the solution $Q$ is
in any sense a {\it ground state} solution. 
It is also not clear that the solution is positive. 
In this section we first establish the existence of a ground state solution; one that minimizes the action subject to a constraint.
We then demonstrate that this minimizer must be our constructed solution. 
In this way we prove Theorem~\ref{theorem2}.

In this section we restrict our nonlinearity and take only 
$f(Q) = |Q|^{p-1} Q$ with $3<p<5$. 
Then the action is 
\begin{align}\label{action}
\mathcal{S}_{\eps,\omega}(u) =   
\frac{1}{2} \|\nabla u \|_{L^2}^2 - \frac{1}{6} \|u\|_{L^6}^6 - 
\frac{\eps}{p+1} \|u\|_{L^{p+1}}^{p+1} + \frac{\omega}{2} \|u\|_{L^2}^2.
\end{align}
We are interested in the constrained minimization problem 
\begin{align} \label{mini}
m_{\eps,\omega} 
:= \inf \{ \mathcal{S}_{\eps,\omega}(u) \; | \;  u \in H^1(\R^3) 
\setminus \{0\},  \; \mathcal{K}_\eps(u) = 0 \}
\end{align}
where 
\begin{align*}
\mathcal{K}_\eps(u) = \frac{d}{d\mu} \mathcal{S}_{\eps,\omega} (T_\mu u  ) \bigg|_{\mu=1} 
= \|\nabla u\|_{L^2}^2 - \|u\|_{L^6}^6 - \frac{3(p-1)}{2(p+1)} \eps \|u\|_{L^{p+1}}^{p+1}
\end{align*}
and $(T_\mu u)(x) = \mu^{3/2} u (\mu x)$ is the $L^2$ scaling operator. 
Note that for $Q_\eps = W + \eta$ as constructed in Theorem \ref{theorem1} we have $\mathcal{K}_\eps(Q_\eps)=0$ 
since any solution to (\ref{elliptic}) will satisfy $\mathcal{K}_\eps(Q) = 0$. 

Before addressing the minimization problem we investigate the implications of our generated solution $Q_\eps$ with specified $\eps$ and corresponding 
$\omega=\omega(\eps)$. 
In particular there is a scaling that generates for us additional solutions to the equation
\begin{align} \label{pureelliptic}
-\Delta Q - Q^5 - \eps |Q|^{p-1} Q + \omega Q = 0
\end{align}
with $3<p<5$.

\begin{remark}\label{scales}
For any $0 < \tilde{\eps} \leq \eps_0$, we have 
solutions to \eqref{pureelliptic} given by
\begin{align*}
  Q^\mu = \mu^{1/2} Q_{\tilde{\eps}}(\mu \cdot)
\end{align*} 
with $\eps = \mu^{(5-p)/2} \tilde{\eps}$ 
and $\omega = \mu^2 \omega(\tilde{\eps})$. So for any $\eps > 0$,
we obtain the family of solutions
\[
  \{ \; Q^\mu \; | \; \mu = \left( \frac{\eps}{\tilde{\eps}} \right)^{\frac{2}{5-p}},
  \; \tilde{\eps} \in (0, \eps_0] \; \} 
\]
with
\[
  \omega =  \left( \frac{\eps}{\tilde{\eps}} \right)^{\frac{4}{5-p}}
  \omega(\tilde{\eps}) \; \in \; \left[ \left( \frac{\eps_0}{\tilde{\eps_0}} 
  \right)^{\frac{4}{5-p}} \omega(\tilde{\eps_0}), \; \infty \right)
\]   
since  as $\tilde{\eps} \downarrow 0$, 
$\left( \frac{\eps}{\tilde{\eps}} \right)^{\frac{4}{5-p}} \omega(\tilde{\eps})
\; \sim \; \tilde{\eps}^{\frac{2(3-p)}{5-p}} \to \infty$.
\end{remark}

We now address the minimization problem by first addressing the existence of a minimizer. 
\begin{lemma} \label{eminimizer}
Take $3<p<5$. 
Let $Q = Q_\eps$ solving \eqref{pureelliptic} with $\omega = \omega(\eps)$ 
be as constructed in Theorem \ref{theorem1}. 
There exists $\eps_0 > 0$ such that for $0<\eps \leq \eps_0$
we have
\begin{align*}
\mathcal{S}_{\eps,\omega(\eps)} (Q_\eps) 
<  \frac{1}{3}\|W\|_{L^6}^6 
= \mathcal{S}_{0,0}(W). 
\end{align*}
It follows, as in Proposition 2.1 of \cite{Slim}, which is in turn based on the earlier 
\cite{Brezis}, 
that the variational problem \eqref{mini} 
with $\omega = \omega(\eps)$ 
admits a positive, radially-symmetric minimizer, which moreover
solves~\eqref{pureelliptic}.
\end{lemma}
\begin{remark}
Strictly speaking, the variational problem~\eqref{mini} produces 
a non-negative minimizer (\cite{Slim}). Since it is then a smooth, 
radially-symmetric solution of~\eqref{pureelliptic}, it follows from 
standard ODE theory that the minimizer is strictly positive.   
\end{remark}

\begin{proof}
We compute directly, ignoring higher order contributions. 
Using (\ref{theotherone}) we write the action as 
\begin{align*}
\mathcal{S}_{\eps,\omega}(Q) &= \frac{1}{3} \int Q^6 + \frac{p-1}{2(p+1)} 
\eps \int |Q|^{p+1} \\
& = \frac{1}{3} \int (W+\eta)^6 + \frac{p-1}{2(p+1)} \eps \int |W+\eta|^{p+1}. 
\end{align*}
Rearranging we have
\begin{align*}
\mathcal{S}_{\eps,\omega}(Q) - \frac{1}{3} \int W^6 = 2 \int W^5 \eta + \frac{p-1}{2(p+1)} \eps \int W^{p+1} + O(\eps^2) 
\end{align*}
where the higher order terms are controlled for $3<p<5$:
\begin{itemize}
\item $\|W^4 \eta^2\|_{L^1} \aleq \|W^4\|_{L^1}\|\eta\|_{L^\infty}^2 \aleq \eps^2 $
\item $\|\eta^6\|_{L^1} \aleq \|\eta\|_{L^6}^6 \aleq \eps^3  $
\item $ \eps\|W^p \eta\|_{L^1} \aleq \eps\|W^p\|_{L^1} \|\eta\|_{L^\infty} \aleq \eps^2$
\item $ \eps \|\eta^{p+1}\|_{L^1} \aleq \eps \|\eta\|_{L^{p+1}}^{p+1} \aleq \eps^{p-1} $. 
\end{itemize} 
We now compute 
\begin{align*}
2 \int W^5 \eta 
&= 2 \left\langle W^5, (H+\lambda^2)^{-1} \left( \eps W^p - \lambda^2 W + N(\eta) \right) \right\rangle  \\
& = 2 \left\langle W^5, (1 + R_0(-\lambda^2)V)^{-1} \bar{P}R_0(-\lambda^2) \left( \eps W^p - \lambda^2 W + N(\eta) \right) \right\rangle
\end{align*}
where we have inserted the definition of $\eta$ from (\ref{eta}) and so identify the two leading order terms. 
There is no problem to also insert the projection $\bar{P}$ from (\ref{projections}) since we have the orthogonality condition (\ref{orth}) by the way we defined $\eps$, $\lambda$, $\eta$. 

We approximate in turn writing only $R_0$ for $R_0(-\lambda^2)$. 
In what follows we use the operators $(1 + VG_0)^{-1}$ and $(1 + VR_0)^{-1}$. The former as acts on the spaces 
\begin{align*}
(1 + VG_0)^{-1} : L^1 \cap (\Lambda W)^{\perp} \to L^1 \cap (1)^{\perp}
\end{align*}
and the later has the expansion 
\begin{align*}
(1 + VR_0)^{-1} = \frac{1}{\lambda} \langle \Lambda W, \cdot \rangle V \Lambda W + O(1) 
\end{align*}
in $L^1$. 
We record here also the adjoint of $\bar{P}$:
\begin{align*}
\bar{P}^* = 1 - P^*, \quad \quad \quad \quad P^* = \frac{\langle \Lambda W, \cdot \rangle}{\int V(\Lambda W)^2} V \Lambda W. 
\end{align*}

To estimate the first term write
\begin{align*}
2 \eps \langle W^5, (1 + R_0V)^{-1} \bar{P} R_0 W^p \rangle 
&= 2 \eps \langle (1 + VR_0)^{-1} W^5,  \bar{P} R_0 W^p \rangle \\
& = 2 \eps \langle (1 + VG_0)^{-1} W^5,  \bar{P} R_0 W^p \rangle + O(\eps^2). 
\end{align*}
The error is controlled with a resolvent identity:   
\begin{align*}
 \eps &\left| \left\langle \left( (1 + VR_0)^{-1} - (1 + VG_0)^{-1} \right) W^5,  \bar{P} R_0 W^p \right\rangle \right| \\
&=  \eps \left| \left\langle (1 + VR_0)^{-1} V (G_0 - R_0) (1 + VG_0)^{-1} W^5, \bar{P} R_0 W^p \right\rangle \right| \\
& =  \eps \left| \left\langle \bar{P}^* (1 + VR_0)^{-1} V (G_0 - R_0) \left( -W^5/4 + V \Lambda W/2 \right) ,  R_0 W^p \right\rangle \right| \\
& \aleq \eps \left\|\bar{P}^* (1 + VR_0)^{-1} V (G_0 - R_0) \left( -W^5/4 + V \Lambda W/2 \right)  \right\|_{L^1} \left\| R_0 W^p \right\|_{L^\infty} \\
& \aleq \eps \left\| V \bar{R} \left( -W^5/4 + V \Lambda W/2 \right) \right\|_{L^1} \|W^p \|_{L^{3/2^-} \cap L^{3/2^+}} \\
& \aleq \eps \lambda \\
& \aleq \eps^2
\end{align*}
where we have written 
\[
  G_0 - R_0 = \lambda G_1 + \bar{R}, \quad
  \bar{R} = \la^2 \tilde{R},
\]
observed that 
$G_1 (-W^5/4 + V \Lambda W/2) = 0$, since $(-W^5/4 + V \Lambda W/2) \perp 1$, and have estimated 
\begin{align*}
\left\| V \bar{R} \left( -W^5/4 + V \Lambda W/2 \right) \right\|_{L^1} 
\aleq \int \langle x \rangle^{-1} dx \int \lambda \frac{|\lambda y|}{\langle \lambda y \rangle} \langle x - y \rangle^{-5} dy \aleq \lambda . 
\end{align*}
Continuing, we have 
\begin{align*}
2 \eps \langle W^5, (1 + R_0V)^{-1} \bar{P} R_0 W^p \rangle 
&= 2 \eps \langle \bar{P}^*(1 + VG_0)^{-1} W^5,  R_0 W^p \rangle + O(\eps^2) \\
&= -\frac{1}{2} \eps \langle W^5, R_0 W^p \rangle + O(\eps^2)
\end{align*}
noting that $ (1 + VG_0)W^5 = W^5 - 5W^4 (-\Delta)^{-1} W^5 = -4W^5 $ since $-\Delta W = W^5$
and that the $\bar{P}^{*}$ can be dropped since $ \Lambda W \perp W^5$ (Remark \ref{nosup}). 
So
\begin{align*}
2 \eps \langle W^5, (1 + R_0V)^{-1} \bar{P} R_0 W^p \rangle 
& = -\frac{1}{2} \eps \langle R_0 W^5, W^p \rangle  + O(\eps^2) \\
& = -\frac{1}{2} \eps \langle G_0 W^5, W^p \rangle + O(\eps^2) \\
& =- \frac{1}{2} \eps \langle W, W^p \rangle + O(\eps^2) \\
& =- \frac{1}{2} \eps \int W^{p+1} + O(\eps^2) 
\end{align*}
where the other error term is bounded:
\begin{align*}
\eps \left| \langle (R_0 - G_0)W^5, W^p \rangle \right|
\aleq \eps \la^2 \left| \langle R_0 G_0 W^5, W^p \rangle \right|
\aleq \eps \la^2 \left| \langle R_0 W, W^p \rangle \right| 
\aleq \eps^2
\end{align*}
observing the computations that produce (\ref{inner2}).  	
Indeed, in (\ref{inner2}) we achieved $ \lambda \langle R_0 (V \psi), W \rangle = \sqrt{3} \int (V \psi) + O(\lambda)$ so our last bound above comes from replacing $V \psi$ with $W^p$. 

For the second term we proceed in a similar manner
\begin{align*}
-2 \la^2 \langle W^5, (1 + R_0V)^{-1} \bar{P} R_0 W \rangle 
&= \frac{1}{2} \la^2 \langle W^5, R_0 W \rangle + O(\eps^{2^-}) \\
&= \frac{1}{2} \la \sqrt{3} \int W^5 + O(\eps^{2^-}) \\
& = 6 \pi \lambda + O(\eps^{2^-}) \\ 
& = - \eps \langle \Lambda W, W^p \rangle + O(\eps^{2^-}) \\
& = \eps \left( \frac{3}{p+1} - \frac{1}{2} \right) \int W^{p+1}+ O(\eps^{2^-})
\end{align*}
where the first equality is just as in the previous computation, and the 
second comes from replacing $V \psi$ in~\eqref{inner2} with $W^5$. 
The error term coming from the difference of the resolvents is similar. Note
\begin{align*}
\la^2 &\left| \left\langle \left( (1 + VR_0)^{-1} - (1 + VG_0)^{-1} \right) W^5,  \bar{P} R_0 W \right\rangle \right| \\
& \aleq \la^2 \left\|\bar{P}^* (1 + VR_0)^{-1} V (G_0 - R_0) \left( -W^5/4 + V \Lambda W/2 \right)  \right\|_{L^1} \left\| R_0 W \right\|_{L^\infty} \\
& \aleq \lambda^3 \left\| R_0 W \right\|_{L^\infty} \\
& \aleq \lambda^3 \lambda^{-1^-} \|W\|_{L^{3^+}} \\
& \aleq \lambda^{2^-} \\
& \aleq \eps^{2^-}. 
\end{align*}

The term coming from $N(\eta)$ is controlled similarly, and so, all together we have
\begin{align*}
\mathcal{S}_{\eps,\omega}(Q) - \frac{1}{3} \int W^6 
&= \left( \frac{3}{p+1} - \frac{1}{2} -\frac{1}{2}   + \frac{p-1}{2(p+1)} \right) \eps \int W^{p+1} + O(\eps^{2^-}) \\
&= -\frac{p-3}{2(p+1)} \eps \int W^{p+1} + O(\eps^{2^-})
\end{align*}
which is negative for $3<p<5$ and $\eps>0$ and small. 
We note that when $p=3$, this leading order term vanishes.
\end{proof}

\begin{lemma}\label{minseq}
Take $3<p<5$. 
Denote by $V=V_\eps$ a non-negative, radially-symmetric
minimizer for~\eqref{mini} with $\omega = \omega(\eps)$
(as established in Lemma \ref{eminimizer}).
Then for any $\eps_j \to 0$, $V_{\eps_j}$
is a minimizing sequence for the (unperturbed) variational problem
\begin{align}\label{0min}
\mathcal{S}_{0,0}(W) =  \min \{ \mathcal{S}_{0,0}(u) \; | \; 
u \in \dot{H}^1 \setminus \{0\}, \; \mathcal{K}_0(u) =0 \}
\end{align}
in the sense that
\[
  \mathcal{K}_0(V_{\eps_j}) \to 0, \qquad  
  \limsup_{\eps \to 0} \mathcal{S}_{0,0}(V_{\eps_j}) 
  \leq \mathcal{S}_{0,0}(W). 
\]
\end{lemma}
\begin{proof}
Since 
\[
  0 = K_\eps(V) = K_0(V) - \frac{3(p-1)}{2(p+1)} \eps \int V^{p+1}, 
\]
and by Lemma~\ref{eminimizer},
\begin{equation} \label{test}
  S_{0,0}(W) > m_{\eps,\omega(\eps)} = S_{\eps,\omega(\eps)}(V) = 
  S_{0,0}(V) - \frac{1}{p+1} \eps \int V^{p+1} + \frac{1}{2} \omega \int V^2,
\end{equation}
the lemma will be implied by the claim:
\begin{equation} \label{toZero}
  \eps \int V^{p+1} \to 0 \quad \mbox{ as } \eps \to 0. 
\end{equation}
To address the claim, first introduce the functional
\[
\begin{split}
 \mathcal{I}_{\eps,\omega}(u) &:= S_{\eps,\omega}(u) - \frac{1}{3} K_\eps(u) \\
 &= \frac{1}{6} \int |\nabla u|^2 + \frac{1}{6} \int |u|^6 + \frac{p - 3}{2(p+1)} \eps \int |u|^{p+1}
 + \frac{1}{2} \omega \int |u|^2 
\end{split}
\]
and observe that since $\mathcal{K}_\eps(V)=0$,
\[
  \mathcal{I}_{\eps,\omega(\eps)}(V) = \mathcal{S}_{\eps,\omega}(V) 
  < \mathcal{S}_{0,0}(W)
\]
and so the following quantities are all bounded uniformly in $\eps$:
\begin{align*}
\int |\nabla V|^2, \;\; \int V^6, \;\; \eps \int V^{p+1}, \;\; \omega \int V^2 \;\; \lesssim 1. 
\end{align*}
By interpolation
\begin{align*}
\eps \int V^{p+1} 
&\leq \eps \|V\|_{L^2}^{(5-p)/2} \|V\|_{L^6}^{3(p-1)/2} \\
& \lesssim \eps \omega^{-(5-p)/4} \left( \omega\int V^2 \right)^{(5-p)/4}.
\end{align*}
So~\eqref{toZero} holds, provided that $\eps^{4/(5-p)} \ll \omega$. 
Since $\omega \sim \eps^2$, this indeed holds for $3<p<5$. 

With the claim in hand we can finish the argument.  
The fact that $\mathcal{K}_0(V) \to 0$ now follows from $\mathcal{K}_\eps(V) = 0$. 
Also, from Lemma \ref{eminimizer} we know that for $\eps \geq 0$
\begin{align*}
\mathcal{S}_{0,0}(V) - \frac{\eps}{p+1} \int V^{p+1} \leq \mathcal{S}_{\eps,\omega}(V) \leq \mathcal{S}_{0,0}(W)
\end{align*}
and so 
$\limsup_{\eps \to 0} \mathcal{S}_{0,0}(V) \leq \mathcal{S}_{0,0}(W)$. 
\end{proof}

\begin{lemma}\label{H1conv}
For a sequence $\eps_j \downarrow 0$, let 
$V=V_{\eps_j}$ be corresponding non-negative, radially-symmetric
minimizers  of \eqref{mini} with $\omega = \omega(\eps_j)$. 
There is a subsequence $\eps_{j_k}$ and a 
scaling $\mu=\mu_{k}$ such that along the subsequence, 
\begin{align*}
V^{\mu} = \mu^{1/2} V (\mu \cdot) \to \nu W
\end{align*}
in $\dot{H}^1$ with $\nu=1$. 
\end{lemma}

\begin{proof}
The result with $\nu=1$ or $\nu=0$ follows from the bubble decomposition of G\'erard \cite{Gerard}
(see eg. the notes of Killip and Vi\c{s}an \cite{Killip}, in particular Theorem 4.7 and the proof of Theorem 4.4). 
Therefore we need only eliminate the possibility that $\nu=0$. 

If $\nu=0$ then $\int |\nabla V_\eps|^2 \to 0$ (along the given subsequence). 
Then by the Sobolev inequality,
\begin{align*}
  0 = \mathcal{K}_\eps(V_\eps) = 
  (1 + o(1)) \int |\nabla V_\eps|^2 - \frac{3(p-1)}{2(p+1)} \eps \int V_\eps^{p+1}, 
\end{align*}
and so 
\begin{align*}
\int |\nabla V_\eps|^2 \aleq \eps \int V_\eps^{p+1}. 
\end{align*}
However, we have already seen
\begin{align*}
\int |\nabla V_\eps|^2 \lesssim \eps \int V_\eps^{p+1} 
\lesssim \eps \omega^{-(5-p)/4} \left( \omega \int V_\eps^2 \right)^{(5-p)/4} \left( \int |\nabla V_\eps|^2 \right)^{3(p-1)/4}
\end{align*}
via interpolation and so 
\begin{align*}
\left( \int |\nabla V_\eps|^2 \right)^{(7-3p)/4} 
\lesssim \eps \omega^{-(5-p)/4} \left( \omega \int V_\eps^2 \right)^{(5-p)/4}
\to 0
\end{align*}
as above. Note that $(7-3p)/4 = -3(p-7/3)/4 <0$. 
Hence $\nu=0$ is impossible and so we conclude that $\nu=1$. 
The result follows. 
\end{proof}
\begin{remark}
This lemma implies in particular that
for $V = V_\eps$, $\omega = \omega(\eps)$, 
$S_{0,0}(V) = S_{0,0}(V^\mu) \to S_{0,0}(W)$,
and so by~\eqref{test} and~\eqref{toZero},
\[
  \omega \int V^2 \;\; \to 0.
\]  
\end{remark}
\begin{remark}
Note that $V^\mu$ is a minimizer of the minimization problem \eqref{mini}, 
and a solution to \eqref{pureelliptic}, with $\eps$ and $\omega$ replaced with
\[ 
  \tilde{\eps} = \mu^{\frac{5-p}{2}} \eps, \qquad \tilde{\omega}=\mu^2 \omega.
\] 
Under this scaling the following properties are preserved:
\begin{align*}
\begin{split}
&\tilde{\eps}^{\frac{4}{5-p}} = \mu^2 \eps^{\frac{4}{5-p}} \ll \mu^2 \omega = \tilde{\omega} \\
&\tilde{\eps} \int (V^{\mu})^{p+1}  = \eps \int V^{p+1} \; \to 0 \\
&\tilde{\omega} \int (V^\mu)^2 = \omega \int V^2 \; \to 0. 
\end{split}
\end{align*}
Moreover,
\[
  \tilde{\eps} \to 0, \quad \tilde{\omega} \to 0,
\]
the latter since otherwise $\| V^\mu \|_{L^2} \to 0$ along some subsequence,
contradicting $V^\mu \to W \not\in L^2$ in $\dot H^1$, 
and then the former by the first relation above.
\end{remark}

\begin{lemma} \label{minconv}
Let 
\begin{align*}
  V^\mu = W + \tilde{\eta}, \quad 
  \|\tilde{\eta}\|_{\dot{H}^1} \to 0, \quad
  \tilde{\eps} \to 0
\end{align*}
be a sequence as provided by Lemma \ref{H1conv}.
There is a further scaling
\begin{align*}
  \nu = \nu_{\tilde{\eps}} = 1 + o(1) 
\end{align*}
so that 
\begin{align*}
\left( V^\mu \right)^\nu = W^\nu + \tilde{\eta}^\nu =: W + \hat\eta 
\end{align*}
retains $\|\hat{\eta}\|_{\dot{H}^1} \to 0$, but also satisfies the orthogonality condition
\begin{align} \label{orthhat}
0=\langle R_0(-\hat{\omega}) V \psi , \mathcal{F}(\hat{\eps}, \hat{\omega}, \hat{\eta}) \rangle 
\end{align}
with the corresponding $\hat{\eps} = \nu^{(5-p)/2} \tilde\eps $ and $\hat{\omega} = \nu^2 \tilde\omega$.
\end{lemma}

\begin{proof}
By the definitions~\eqref{preinvert} of 
$\mathcal{F}(\hat{\eps}, \hat{\omega}, \hat{\eta})$, and~\eqref{Vpsi_int} 
of $\psi$, we may rewrite the above inner-product as
\[
\langle R_0(-\hat{\omega}) V \psi , \mathcal{F}(\hat{\eps}, \hat{\omega}, \hat{\eta}) \rangle
= -\frac{5}{\sqrt{3}} \langle \hat{\eta}, (H + \hat{\omega}) R_0(-\hat{\omega}) W^4 \Lambda W \rangle 
\]
and observe from the resonance equation~\eqref{res_un_norm}
\[
  5 R_0(-\hat{\omega}) W^4 \Lambda W = \Lambda W - \hat{\omega} R_0(-\hat{\omega})
  \Lambda W
\]
and so 
\[
\begin{split}
  (H + \hat{\omega}) R_0(-\hat{\omega}) W^4 \Lambda W &=
  (-\Delta + \hat{\omega} - 5 W^4) R_0(-\hat{\omega}) W^4 \Lambda W \\ &=
  (1 - 5 W^4 R_0(-\hat{\omega})) W^4 \Lambda W = \hat{\omega} W^4 R_0(-\hat{\omega})
  \Lambda W
\end{split}
\]
so the desired orthogonality condition reads
\[
0 = \frac{1}{\sqrt{\hat{\omega}}} \langle \hat{\eta}, (H + \hat{\omega}) R_0(-\hat{\omega}) 
W^4 \Lambda W \rangle 
= \sqrt{\hat{\omega}} \langle W^{\nu} - W + \tilde\eta^\nu, W^4 R_0(-\hat{\omega}) 
\Lambda W \rangle.
\]
Now since $\Lambda W = \frac{d}{d \mu} W^\mu |_{\mu=1}$, by Taylor expansion
\[
  \| W^\nu - W - (\nu-1) \Lambda W \|_{L^6} \aleq (\nu-1)^2,
\]
and using~\eqref{kernelbddw}
\[
  \| W^4 R_0(-\hat{\omega}) \Lambda W \|_{L^{\frac{6}{5}}} \aleq
  \| R_0(-\hat{\omega}) \Lambda W \|_{L^\infty} \aleq
  \frac{1}{\sqrt{\hat{\omega}}},
\]
we arrive at
\[
  0 = (\nu - 1) \left( \sqrt{\hat{\omega}} \langle \Lambda W, \; W^4 R_0(-\hat{\omega}) 
  \Lambda W \rangle \right)
  + O((\nu-1)^2) + O(\| \tilde{\eta}^\nu \|_{L^6})
\]
Computations exactly as for (\ref{inner2}) lead to
\[
  \sqrt{\hat{\omega}} \langle \Lambda W, \; W^4 R_0(-\hat{\omega}) 
  \Lambda W \rangle = \frac{6 \sqrt{3} \pi}{5} + O(\sqrt{\hat{\omega}}),
\] 
and so the desired orthogonality condition reads
\[
  0 = (\nu - 1) (1 + o(1)) 
  + O((\nu-1)^2) + O(\| \tilde{\eta}^\nu \|_{L^6})
\]
which can therefore be solved for $\nu=1+o(1)$
using $\| \tilde{\eta}^\nu \|_{L^6} = o(1)$.
\end{proof}
The functions
\[
  W_{\hat{\eps}} := (V^\mu)^\nu = W + \hat{\eta}
\]
produced by Lemma~\ref{minconv} solve the minimization problem (\ref{mini}), 
and the PDE (\ref{pureelliptic}), with $\eps$ and $\omega$ replaced (respectively) 
by $\hat{\eps} \to 0$ and $\hat{\omega}$. 
Since $\nu_{\tilde{\eps}} = 1 + o(1)$, the properties
\[
  \hat{\eps}^{\frac{4}{5-p}} \ll \hat{\omega} \to 0, \quad
  \hat{\eps} \int W_{\hat{\eps}}^{p+1} \to 0, \quad
  \hat{\omega} \int W_{\hat{\eps}}^2 \to 0
\]
persist.

It remains to show that that $V_{\hat{\eps}}$ agrees with $Q_{\hat{\eps}}$ constructed in 
Theorem~\ref{theorem1}. First:
\begin{lemma}\label{etabounds}
For $3 < r \leq \infty$, and $\hat{\eps}$ sufficiently small,
\[
  \|\hat{\eta} \|_{L^r} \lesssim \hat{\eps} + \sqrt{\hat{\omega}}^{1-\frac{3}{r}}. 
\]
\end{lemma}
\begin{proof}
Since $W_{\hat{\eps}}$ is a solution of (\ref{elliptic}), the 
remainder $\hat{\eta}$ must satisfy (\ref{eta}). So
\begin{align*}
\|\hat{\eta}\|_{L^r} &= \|(H+\hat{\omega})^{-1} \left( -\hat{\omega} W + 
\hat{\eps} f(W) + N(\hat{\eta}) \right) \|_{L^r} \\
&\lesssim \hat{\eps} + \sqrt{\hat{\omega}}^{1-\frac{3}{r}}
+ \|R_0(-\hat{\omega}) N(\hat{\eta})\|_{L^r} 
\end{align*}
using~\eqref{orthhat} and
after observing the computations of Lemma \ref{etalemma}. 
We now establish the required bounds on the remainder,
beginning with $3 < r < \infty$. Let $q \in (1, \frac{3}{2})$ satisfy
$\frac{1}{q} - \frac{1}{r} = \frac{2}{3}$:
\begin{itemize} 
\item 
$ \ds
\| R_0(-\hat{\omega}) \hat{\eps} W^{p-1} \hat{\eta}\|_{L^r}  
\lesssim \hat{\eps} \| W^{p-1} \hat{\eta}\|_{L^q} 
\lesssim \hat{\eps} \|W\|_{L^{\frac{3}{2}(p-1)}}^{p-1} \| \hat{\eta} \|_{L^r} 
\lesssim \hat{\eps} \|\hat{\eta}\|_{L^r}
$
\item 
$ \ds
\|R_0(-\hat{\omega}) \hat{\eps} \hat{\eta}^p \|_{L^r} 
\lesssim \hat{\eps} \hat{\omega}^{-\frac{5-p}{4}} \|\hat{\eta}^p \|_{L^{\frac{6r}{6 + (p-1)r}}}
\lesssim o(1) \|\hat{\eta}\|_{L^6}^{p-1} \|\hat{\eta}\|_{L^r}
\lesssim o(1) \| \hat{\eta} \|_{L^r} 
$
\item
$\ds 
\| R_0(-\hat{\omega}) W^3 \hat{\eta}^5\|_{L^r} 
\lesssim \|W^3 \hat{\eta}^2\|_{L^q} 
\lesssim \|W\|_{L^6}^3 \|\hat{\eta}\|_{L^6} \|\hat{\eta}\|_{L^r} 
\lesssim o(1) \|\hat{\eta}\|_{L^r} 
$
\item
$\ds
\| R_0(-\hat{\omega}) \hat{\eta}^5 \|_{L^r}
\lesssim \| \hat{\eta}^5 \|_{L^q}
\lesssim \| \hat{\eta}  \|_{L^6}^4  \| \hat{\eta} \|_{L^r}
\lesssim o(1) \| \hat{\eta} \|_{L^r} 
$
\end{itemize}
where in the second inequality we used $\hat{\omega} \gg \hat{\eps}^{4/(5-p)}$.
Combining, we have achieved 
\begin{align*}
\|\hat{\eta}\|_{L^r} \lesssim 
\hat{\eps} + \sqrt{\hat{\omega}}^{1-\frac{3}{r}} + o(1) \|\hat{\eta}\|_{L^r}
\end{align*}
and so obtain the desired estimate for $3 < r < \infty$.
It remains to deal with $r=\infty$.
The first three estimates proceed similarly, while the last one uses the 
now-established $L^r$ estimate:
\begin{itemize} 
\item 
$ \ds
\|R_0(-\hat{\omega}) \hat{\eps} W^{p-1} \hat{\eta}\|_{L^\infty}  
\lesssim \hat{\eps} \| W^{p-1} \hat{\eta} \|_{L^{\frac{3}{2}-}\cap L^{\frac{3}{2}+}} \\
\lesssim \hat{\eps} \|W\|_{L^{\frac{3}{2}(p-1)-} \cap L^{\frac{3}{2}(p-1)+}}^{p-1} \|\hat{\eta}\|_{L^\infty} 
\lesssim \hat{\eps} \|\hat{\eta}\|_{L^\infty}
$
\item 
$ \ds
\|R_0(-\hat{\omega}) \hat{\eps} \hat{\eta}^p \|_{L^\infty}
\lesssim \hat{\eps} \hat{\omega}^{-\frac{5-p}{4}} \| \hat{\eta}^p \|_{L^\frac{6}{p-1}} 
\lesssim o(1)  \|\hat{\eta}\|_{L^6}^{p-1} \|\hat{\eta}\|_{L^\infty} 
\lesssim o(1) \|\hat{\eta}\|_{L^\infty} 
$
\item
$\ds 
\| R_0(-\hat{\omega}) W^3 \hat{\eta}^5\|_{L^\infty} 
\lesssim \|W^3 \hat{\eta}^2\|_{L^{\frac{3}{2}-} \cap L^{\frac{3}{2}+}} 
\lesssim \|W\|_{L^{6-} \cap L^{6+}}^3 \|\hat{\eta}\|_{L^6} \|\hat{\eta}\|_{L^\infty} \\
\lesssim o(1) \|\hat{\eta}\|_{L^\infty} 
$
\item
$\ds
\|R_0(-\hat{\omega}) \hat{\eta}^5\|_{L^\infty}
\lesssim \| \hat{\eta}^5 \|_{L^{\frac{3}{2}-} \cap L^{\frac{3}{2}+}}
\lesssim \|  \hat{\eta} \|_{L^{6-} \cap L^{6+}}^4 \| \hat{\eta} \|_{L^\infty} \\
\lesssim (\hat{\eps} + \hat{\omega}^{\frac{1}{4}-})^4 \| \hat{\eta} \|_{L^\infty}
\lesssim o(1) \| \hat{\eta} \|_{L^\infty}
$
\end{itemize}
which, combined, establish the desired estimate with $r = \infty$.
Strictly speaking, these are {\it a priori} estimates, since we do not
know $\hat{\eta} \in L^r$ for $r > 6$ to begin with. 
However, the typical argument of performing the estimates on a series of smooth functions that approximate $\eta$ remedies this after passing to the limit. 
\end{proof}
  
\begin{lemma}
Write $\hat{\omega} = \hat{\lambda}^2$.
For $\hat{\eps}$ sufficiently small, $\| \hat{\eta} \|_{L^\infty} \aleq \hat{\eps}$, and
$\hat{\lambda} = \lambda(\hat{\eps},\hat{\eta})$ as given in Lemma~\ref{lambdalemma}.
Moreover, $W_{\hat{\eps}} = W + \hat{\eta} = Q_{\hat{\eps}}$.
\end{lemma}
\begin{proof}
From the orthogonality equation~\eqref{orthhat},
\[
\begin{split}
  0 &= \langle R_0(-\hat{\lambda}^2) V \psi,
  -\hat{\lambda}^2 W + \hat{\eps} W^p + N(\hat{\eta}) \rangle \\
  &= -\hat{\lambda} \cdot \hat{\lambda} \langle R_0(-\hat{\lambda}^2) V \psi, W \rangle
  + \hat{\eps} \langle R_0(-\hat{\lambda}^2) V \psi, W^p \rangle
  + \langle R_0(-\hat{\lambda}^2) V \psi, N(\hat{\eta}) \rangle.
\end{split}
\]
Now re-using estimates~\eqref{inner1} and~\eqref{inner2}, as well as
\begin{itemize}
\item
$ \ds
\begin{aligned}[t]
|\langle R_0(-\hat{\lambda}^2) V \psi, W^3 \hat{\eta}^2 \rangle | 
&\aleq \|R_0(-\hat{\lambda}^2) V \psi \|_{L^6} \|W^3 \hat{\eta}^2 \|_{L^{6/5}} \\
&\aleq \|V \psi\|_{L^{6/5}} \|\hat{\eta}\|_{L^\infty}^2 \|W^3\|_{L^{6/5}} 
\aleq \| \hat{\eta} \|_{L^\infty}^2
\end{aligned}
$
\item
$ \ds
\begin{aligned}[t]
|\langle R_0(-\hat{\lambda}^2) V \psi, \hat{\eta}^5 \rangle | 
&\aleq \|R_0(-\hat{\lambda}^2) V \psi \|_{L^6} \|\hat{\eta}^5\|_{L^\frac{6}{5}} \\
&\aleq \|V \psi\|_{L^\frac{6}{5}} \|\hat{\eta}\|_{L^6}^5
\aleq  \|\hat{\eta}\|_{L^6}^5
\end{aligned}
$
\item
$ \ds 
\begin{aligned}[t]
|\langle R_0(-\hat{\lambda}^2) V \psi, \hat{\eps} \hat{\eta}^{p} \rangle | 
&\aleq \hat{\eps} \|R_0(-\hat{\lambda}^2) V\psi\|_{L^6} \|\hat{\eta}^p\|_{L^\frac{6}{5}} \\
& \aleq \hat{\eps} \|V \psi\|_{L^\frac{6}{5}} \|\hat{\eta} \|_{L^{\frac{6}{5}p}}^p
\aleq \hat{\eps} \cdot \|\hat{\eta} \|_{L^{\frac{6}{5}p}}^p
\end{aligned}
$
\item
$ \ds 
\begin{aligned}[t]
|\langle R_0(-\hat{\lambda}^2) V \psi, \hat{\eps} W^{p-1} \hat{\eta} \rangle | 
&\aleq \hat{\eps} \|R_0(-\hat{\lambda}^2) V\psi\|_{L^6} \|W^{p-1} \hat{\eta}\|_{L^\frac{6}{5}} \\
& \aleq \hat{\eps} \|V \psi\|_{L^\frac{6}{5}} \|W^{p-1} \|_{L^{\frac{3}{2}}} \| \hat{\eta} \|_{L^6}
\aleq \hat{\eps} \| \hat{\eta} \|_{L^6},
\end{aligned}
$
\end{itemize}
combined with Lemma~\ref{etabounds}, yields
\[
  (\hat{\lambda} - \lambda^{(1)} \hat{\eps})(1 + O(\hat{\lambda}^{1-}))
  = O( \hat{\lambda}^2 + \hat{\eps}^2 + \hat{\eps} \hat{\lambda}^{\frac{1}{2}})
\]
from which follows
\[
  \hat{\lambda} - \lambda^{(1)} \hat{\eps} \ll \hat{\eps},
\]
and then by Lemma~\ref{etabounds} again,
\[
  \| \hat{\eta} \|_{L^r} \aleq \hat{\eps}^{1-\frac{3}{r}}, \quad 3 < r \leq \infty.
\]
It now follows from Lemma~\ref{lambdalemma}
that $\hat{\lambda} = \lambda(\hat{\eps}, \hat{\eta})$ for 
$\hat{\eps}$ small enough.

Finally, the uniqueness of the fixed-point in the $L^\infty$-ball of radius $R \hat{\eps}$
from Lemma~\ref{etalemma} implies that
$W_{\hat{\eps}} = Q_{\hat{\eps}}$, where $Q_{\hat{\eps}}$
is the solution constructed in Theorem~\ref{theorem1}.
\end{proof}

We have so far established that, up to subsequence, and rescaling, 
a sequence of minimizers $V_{\eps_j}$ eventually coincides with a
solution $Q_\eps$ as constructed in Theorem~\ref{theorem1}:
$\xi_j^{1/2} V_{\eps_j}(\xi_j \cdot)= Q_{\hat{\eps}_j}$
(here $\xi_j = \nu_j \mu_j$). 
It remains to remove the scaling $\xi_j$ and establish that $\hat{\eps}_j = \eps_j$:
\begin{lemma}\label{trivscale}
Suppose $V(x) = \xi^{-\frac{1}{2}} Q_{\hat{\eps}}(x/\xi)$ solves~\eqref{pureelliptic}
with $\omega=\omega(\eps)$ (as given in Theorem \ref{theorem1}), where 
$\hat\eps = \xi^{(5-p)/2}\eps$, $\hat\omega = \xi^{2} \omega$, and 
$\hat\omega = \omega(\hat\eps)$. Then $\xi=1$ and $\hat\eps=\eps$, and so 
$V =Q_\eps$. 
\end{lemma}
\begin{proof}
By assumption $\hat{\omega} = \xi^2 \omega(\eps) = \omega(\hat{\eps})$, so
\begin{align*}
\omega(\eps) = \Omega_\eps(\hat{\eps}), \qquad 
\Omega_\eps(\hat\eps) := \left( \frac{\eps}{\hat\eps}\right)^{4/(5-p)} \omega(\hat\eps). 
\end{align*}
This relation is satisfied if $\hat{\eps} = \eps$ ($\xi = 1$), and our goal is to show
it is not satisfied for any other value of $\hat{\eps}$. 
Thus we will be done if we can show that $\Omega_\eps$ is monotone in $\hat\eps$.
Take $\eps_1$ and $\eps_2$ with $0<\eps_1<\eps_2\leq \eps_0$. 
Let $\alpha = 4/(5-p)>2$ and assume that $0<\eps_2-\eps_1 \ll \eps_1$. 
Denoting $\omega(\eps_j) = \lambda_j^2$, we estimate: 
\begin{align*}
\eps^{-\alpha} \left( \Omega_\eps(\eps_2) - \Omega_\eps(\eps_1) \right)
&= \eps_2^{-\alpha}\lambda^2_2 - \eps_1^{-\alpha} \lambda^2_1  \\
&= \eps_2^{-\alpha}(\lambda_2-\lambda_1)(\lambda_2+\lambda_1)
 + \lambda^2_1 \left( \eps_2^{-\alpha} - \eps_1^{-\alpha}\right) \\
& \approx \eps_1^{-\alpha} \lambda^{(1)}(\eps_2-\eps_1) \cdot 2\lambda^{(1)}\eps_1  \\
& \quad + \eps_1^2 (\lambda^{(1)})^2  \eps_1^{-\alpha}\left(-\frac{\alpha}{\eps_1}(\eps_2-\eps_1)
+O\left( \left( \frac{\eps_2-\eps_1}{\eps_1} \right)^2 \right)  \right)    \\
&\approx \eps_1^{1-\alpha} (\lambda^{(1)})^2 (\eps_2-\eps_1) \left( 2 -\alpha  \right) \\
&<0 
\end{align*}
where we have used Lemma \ref{monotone}. 
With the monotonicity argument complete we conclude that $\eps=\hat\eps$ and $\xi=1$ so there follows $V  = Q_\eps$. 
\end{proof}

The remaining lemma completes the proof of Theorem~\ref{theorem2}:
\begin{lemma} \label{unique}
There is $\eps_0>0$ such that for $0<\eps \leq \eps_0$ and $\omega=\omega(\eps)$,
the solution $Q_\eps$ of~\eqref{pureelliptic} constructed in Theorem \ref{theorem1}
is the unique positive, radially symmetric solution of the  minimization problem \eqref{mini}.
\end{lemma}
\begin{proof}
This is the culmination of the previous series of Lemmas. 
We know that minimizers $V=V_\eps$ exist by Lemma \ref{eminimizer}. 
Arguing by contradiction, if the statement is false, there is a sequence $V_{\eps_j}$,
$\eps_j \to 0$, of such minimizers, for which $V_{\eps_j} \not= Q_{\eps_j}$. 
We apply Lemmas \ref{minseq}, \ref{H1conv}, \ref{minconv}, \ref{etabounds} and \ref{trivscale} in succession to this sequence, to conclude that along a subsequence,
$V_{\eps_j}$ and $Q_{\eps_j}$ eventually agree, a contradiction.
\end{proof}

Finally, for a given $\eps$, we establish a range of $\omega$ for which a minimizer exists and is, up to scaling, a constructed solution. This addresses Remark \ref{range_remark}. 
\begin{corollary}\label{range}
Fix $\eps > 0$ and take $\omega \in [\underline{\omega}, \infty)$ where
\begin{align*}
\underline{\omega} = \eps^{4/(5-p)} \eps_0^{-4/(5-p)} \omega(\eps_0) \leq \omega(\eps). 
\end{align*}
The minimization problem \eqref{mini} with $\eps$ and $\omega$ has a solution $Q$ given by 
\begin{align*}
Q(x) = \mu^{1/2} Q_{\hat\eps}(\mu x) 
\end{align*}
where $Q_{\hat\eps}$ is a constructed solution with $0 < \hat\eps \leq \eps_0$ and corresponding $\omega(\hat\eps)$. The scaling factor, $\mu$, satisfies the relationships 
\begin{align*}
\eps = \hat\eps \mu^{(5-p)/2}, \quad \quad \omega = \omega(\hat\eps) \mu^2. 
\end{align*}
\end{corollary}

\begin{proof}
Fix $\eps > 0$. Take any $0 < \hat\eps \leq \eps_0$ and corresponding constructed $\omega(\hat\eps)$ and constructed solution $Q_{\hat\eps}$. Then, for scaling $\mu = (\eps/\hat\eps)^{2/(5-p)}$ the function   
\begin{align*}
Q(x) = \mu^{1/2} Q_{\hat\eps}(\mu x)
\end{align*}
is a solution to the elliptic problem (\ref{pureelliptic}) with $\eps$ and $\omega = \omega(\hat\eps) \mu^2$. Recall from Lemma \ref{trivscale} that $\omega(\hat\eps) \mu^2$ is monotone in $\hat\eps$. Taking $\hat\eps \downarrow 0$ yields $\omega \to \infty$. Setting $\hat\eps = \eps_0$ yields $\omega = \underline{\omega}$. 

In other words if we fix $\eps$ and $\omega \in [\underline\omega, \infty)$ from the start we determine an $\hat\eps$ and $\mu$ that generate the desired $Q$. We claim that the function $Q(x)$ is a minimizer of the problem (\ref{mini}) with $\eps$ and $\omega$. Suppose not. That is, suppose there exists a function $0 \neq v \in H^1$ with $\mathcal{K}_\eps(v) = 0$ such that $\mathcal{S}_{\eps, \omega}(v) < \mathcal{S}_{\eps,\omega}(Q)$. Set $w(x) = \mu^{-1/2} v(\mu^{-1} x)$ and note that $0 = \mathcal{K}_\eps(v) = \mathcal{K}_{\hat\eps}(w)$. 
We now see
\begin{align*}
\mathcal{S}_{\hat\eps, \omega(\hat\eps)}(w) = \mathcal{S}_{\eps, \omega}(v) < \mathcal{S}_{\eps,\omega}(Q) = \mathcal{S}_{\hat\eps, \omega(\hat\eps)}(Q_{\hat\eps})
\end{align*}
which contradicts the fact that $Q_{\hat\eps}$ is a minimizer of the problem (\ref{mini}) with $\hat\eps$ and $\omega(\hat\eps)$. 
Therefore, $Q(x)$ is a minimizer of (\ref{mini}) with $\eps$ and $\omega$, which concludes the proof.  
\end{proof}


\section{
Dynamics Below the Ground States
} \label{dynamics}

In this final section we establish Theorem~\ref{dichotomy}, the scattering/blow-up 
dichotomy for the perturbed critical NLS~\eqref{NLS}. 

We begin by summarizing the local existence theory for~\eqref{NLS}.
This is based on the classical {\it Strichartz estimates} for the solutions 
of the homogeneous linear Schr\"odinger equation
\[
  i \p_t u = -\Delta u, \;\; u|_{t=0} = \phi \in L^2(\R^3) 
  \; \implies \; 
  u(x,t) = e^{i t \Delta} \phi \; \in C(\R,L^2(\R^3))
\]
and the inhomogeneous linear Schr\"odinger equation (with zero initial data)
\[
  i \p_t u = -\Delta u + f(x,t), \;\; u|_{t=0} = 0
  \; \implies \; 
  u(x,t) = -i \int_0^t e^{i (t-s) \Delta} f(\cdot,s) ds \; :
\]
\begin{equation} \label{Strichartz}
  \| e^{i t \Delta} \phi \|_{S(\R)} \leq C \| \phi \|_{L^2(\R^3)}, \quad
  \left\| \int_0^t e^{i(t-s)\Delta} f(\cdot,s) ds  \right\|_{S(I)} \leq C \| f \|_{N(I)},  
\end{equation}
where we have introduced certain Lebesgue norms for space-time functions
$f(x,t)$ on a time interval $t \in I \subset \R$:
\[
\begin{split}
  & \qquad \| f \|_{L^r_t L^q_x (I)} = \left \| \| f(\cdot,t) \|_{L^q(\R^3)} \right\|_{L^r(I)}, \\
  & \| f \|_{S(I)} := \| f \|_{L^\infty_t L^2_x (I) \cap L^2_t L^6_x (I)}, \quad
  \| f \|_{N(I)} := \| f \|_{L^1_t L^2_x (I) + L^2_t L^{\frac{6}{5}}_x (I)},
\end{split}
\]
together with the integral (Duhamel) reformulation of the Cauchy problem \eqref{NLS}:
\[
  u(x,t) = e^{i t \Delta} u_0 + i \int_0^t e^{i(t-s) \Delta} ( |u|^4 u + \e |u|^{p-1} u ) ds
\]
which in particular gives the sense in which we consider $u(x,t)$ to be a {\it solution} of~\eqref{NLS}.
This lemma summarizing the local theory is standard (see, for example~\cite{Cazenave,KOPV}):   
\begin{lemma} \label{local}
Let $3 \leq p < 5$, $\eps > 0$.
Given $u_0 \in H^1(\R^3)$, there is a unique solution
$u \in C( (-T_{min}, T_{max}) ; H^1(\R^3) )$ of~\eqref{NLS}
on a maximal time interval $I_{max} = (-T_{min}, T_{max}) \ni 0$. Moreover:
\begin{enumerate}
\item
space-time norms: 
$u, \nabla u \in S(I)$ for each compact time interval $I \subset I_{max}$; 
\item
blow-up criterion:
if $T_{max} < \infty$, then $\| u \|_{L^{10}_t L^{10}_x ([0,T_{max}))} = \infty$
(with similar statement for $T_{min}$); 
\item
scattering:
if $T_{max} = \infty$ and $\| u \|_{L^{10}_t L^{10}_x ([0,\infty))} < \infty$,
then $u$ scatters (forward in time) to $0$ in $H^1$:
\[
  \exists \; \phi_+ \in H^1(\R^3) \mbox{ s.t. }
  \| u(\cdot, t) - e^{i t \Delta} \phi_+ \|_{H^1} \to 0 \mbox{ as } t \to \infty
\]
(with similar statement for $T_{min}$); 
\item
small data scattering: for $\| u_0 \|_{H^1}$ sufficiently small,
$I_{max} = \R$,  \newline
$\| u \|_{L^{10}_t L^{10}_x (\R)} \aleq \| \nabla u_0 \|_{L^2}$,
and $u$ scatters (in both time directions).
\end{enumerate}
\end{lemma}
\begin{remark}
The appearance here of the $L^{10}_t L^{10}_x$ space-time norm is natural
in light of the Strichartz estimates~\eqref{Strichartz}. Indeed, interpolation
between $L^\infty_t L^2_x$ and $L^2_t L^6$ shows that 
\[
  \| e^{i t \Delta} \phi \|_{L^r_t L^q_x (\R)} \aleq \| \phi \|_{L^2},
  \quad \frac{2}{r} + \frac{3}{q} = \frac{3}{2}, \quad 2 \leq r \leq \infty
\]
(such an exponent pair $(r,q)$ is called {\it admissible}), so
then if $\nabla \phi \in L^2$, by a Sobolev inequality,
\[
  \| e^{i t \Delta} \phi \|_{L^{10}_x} \aleq \| \nabla e^{i t \Delta} \phi \|_{L^{\frac{30}{13}}_x}
  \in L^{10}_t,
\]
since $(r = 10, q = \frac{30}{13})$ is admissible.
\end{remark}
The next lemma is a standard extension of the local theory called a {\it perturbation} or 
{\it stability} result, which shows that any `approximate solution' has an actual solution remaining 
close to it. In our setting (see \cite{TVZ,KOPV}):
\begin{lemma} \label{perturbations}
Let $\tilde{u} : \R^3 \times I \to \C$ be defined on time interval $0 \in I \subset \R$ with
\[
  \| \tilde{u} \|_{L^\infty_t H^1_x (I) \cap L^{10}_t L^{10}_x (I)} \leq M,
\]
and suppose $u_0 \in H^1(\R^3)$ satisfies $\| u_0 \|_{L^2} \leq M$.
There exists $\delta_0 = \delta_0(M) > 0$ such that if for any $0 < \delta < \delta_0$,
$\tilde{u}$ is an approximate solution of~\eqref{NLS} in the sense
\[
  \| \nabla e \|_{L^{\frac{10}{7}}_t L^{\frac{10}{7}}_x (I)} \leq \delta, \qquad
  e := i \p_t \tilde{u} + \Delta \tilde{u} + |\tilde{u}|^4 \tilde{u} + \e |\tilde{u}|^{p-1} \tilde{u},
\]
with initial data close to $u_0$ in the sense
\[
  \| \nabla \left( \tilde{u}(\cdot,0) - u_0 \right) \|_{L^2} \leq \delta,
\]
then the solution $u$ of~\eqref{NLS} with initial data $u_0$ has $I_{max} \supset I$, and
\[
  \| \nabla \left( u - \tilde{u} \right) \|_{S(I)} \leq C(M) \delta.
\]
\end{lemma}
\begin{remark}
The space-time norm $\nabla e \in L^{\frac{10}{7}}_t L^{\frac{10}{7}}_x$ in which the error is
measured is natural in light of the Strichartz estimates~\eqref{Strichartz}, since 
$L^{\frac{10}{7}}_t L^{\frac{10}{7}}_x$ is the dual space of  $L^{\frac{10}{3}}_t L^{\frac{10}{3}}_x$,
and $(\frac{10}{3}, \frac{10}{3})$ is an admissible exponent pair. 
\end{remark}

Given a local existence theory as above, an obvious next problem is to determine if the 
solutions from particular initial data $u_0$ are {\it global} ($I_{max} = \R$), or
exhibit {\it finite-time blow-up} ($T_{max} < \infty$ and/or $T_{min} < \infty$).
Theorem~\ref{dichotomy} solves this problem for radially-symmetric 
initial data lying `below the ground state' level of the action: 
for any $\eps > 0$, $\omega > 0$, set
\begin{equation} \label{inf}
  m_{\eps,\omega} := \inf \{ \mathcal{S}_{\eps,\omega}(u) \; |  \; 
  u \in H^1(\R^3) \setminus \{0\}, \mathcal{K}_\eps(u) = 0 \}
\end{equation}
(see~\eqref{minSp} for expressions for the functionals 
$\mathcal{S}_{\eps,\omega}$ and $\mathcal{K}_\eps$),
and note that for $\e \ll 1$ and $\om = \om(\e)$, by Theorem~\ref{theorem2} 
we have $m_{\e,\om} = S_{\e,\om}(Q_\e)$.
From here on, we fix a choice of
\[
  \e > 0, \quad \om > 0, \quad p \in (3,5)
 \]
(though some results discussed below extend to $p \in (\frac{7}{3}, 5)$):
\begin{theorem} \label{dichotomy2}
Let $u_0 \in H^1(\R^3)$ be radially-symmetric and satisfy
\[
  \mathcal{S}_{\eps,\omega} (u_0) < m_{\eps,\omega},
\]
and let $u$ be the corresponding solution to~\eqref{NLS}:
\begin{enumerate}
\item
If $\mathcal{K}_{\eps}(u_0) \geq 0$, $u$ is global, and scatters to $0$ as $t \to \pm \infty$; 
\item 
if $\mathcal{K}_{\eps}(u_0) < 0$, $u$ blows-up in finite time (in both time directions).
\end{enumerate}
\end{theorem}
\begin{remark}
The argument which gives the finite-time blow-up (the second statement) is classical,
going back to~\cite{OgawaTsutsumi, Nawa}. It rests on the following ingredients:
conservation of mass and energy imply
$\S_{\e,\om}(u) \equiv \S_{\e,\om}(u_0) < m_{\e,\om}$,
so that the condition $\K_\e(u) < 0$ is preserved (by definition of $m_{\e,\om}$);
a spatially cut-off version of the formal {\it variance identity} for (NLS)
\begin{equation} \label{virial}
  \frac{d^2}{dt^2} \frac{1}{2} \int |x|^2 |u(x,t)|^2 dx =
  \frac{d}{dt} \int x \cdot \Im \left( \bar{u} \nabla u \right) dx = 
  2 \K_\e (u) \; ;
\end{equation}
and exploitation of radial symmetry to control the errors introduced by the cut-off.
In fact, a complete argument in exactly our setting is given as the proof 
of Theorem 1.3 in~\cite{Slim} (it is stated there for dimensions $\geq 4$
but in fact the proof covers dimension $3$ as well).
So we will focus here only on the proof of the first (scattering) statement.  
\end{remark}
The concentration-compactness approach of Kenig-Merle~\cite{Kenig} 
to proving the scattering statement is by now standard. 
In particular, \cite{Slim2} provides a complete proof for the analogous problem
in dimensions $\geq 5$. In fact, the proof there is more complicated for two reasons:
there is no radial symmetry restriction; and in dimension $n$, the corresponding
nonlinearity includes the term $|u|^{p-1} u$ with $p > 1 + \frac{4}{n}$ loses
smoothness, creating extra technical difficulties. 
We will therefore provide just a sketch of the (simpler) argument for our case, 
closely following~\cite{KOPV}, where this approach is implemented for the {\it defocusing} 
quintic NLS perturbed by a cubic term, and taking the additional variational arguments we need
here from~\cite{Slim,Slim2}, highlighting points where modifications are needed.  

In the next lemma we recall some standard variational estimates for functions with 
action below the ground state level $m_{\e,\om}$. 
The idea goes back to~\cite{PayneSattinger},
but proofs in this setting are found in~\cite{Slim,Slim2}.
Recall the `unperturbed' ground state level is attained by the Aubin-Talenti function $W$:
\[
  m_{0,0} := \E_0(W) = \inf \{ \mathcal{E}_{0}(u) \; |  \; 
  u \in H^1(\R^3) \setminus \{0\}, \mathcal{K}_0(u) = 0 \},
\]
and introduce the auxilliary functional
\[
\begin{split}
  \mathcal{I}_\om(u) & := \S_{\e,\om}(u) - \frac{2}{3(p-1)} \K_\e(u) \\ &=
  \frac{p-\frac{7}{3}}{2(p-1)} \int |\nabla u|^2 + \frac{5-p}{6(p-1)} \int |u|^6
  + \frac{1}{2} \om \int |u|^2
\end{split}
\]
which is useful since all its terms are positive, and note
\begin{equation} \label{H1bound}
  \K_\e(u) \geq 0 \; \implies \; 
  \| u \|_{H^1}^2 \aleq \mathcal{I}_\om(u) \leq \S_{\e,\om}(u).
\end{equation}
Define, for $0 < m^* < m_{\e,\om}$, the set
\[
  \mathcal{A}_{m^*} :=  
  \{ u \in H^1(\R^3) \; | \; \S_{\e,\om}(u) \leq m^*, \; K_\e(u) > 0 \}
\]
and note that it is is preserved by~\eqref{NLS}:
\[
  u_0 \in \mathcal{A}_{m^*} \; \implies  u(\cdot,t) \in \mathcal{A}_{m^*}
  \mbox{ for all } t \in I_{max}.
\]
Indeed, by conservation of mass and energy
$\S_{\e,\om}(u(\cdot,t)) = \S_{\e,\om}(u_0) \leq m^*$. Moreover if for some 
$t_0 \in I_{max}$, $\K_\e(u(\cdot,t_0)) \leq 0$, then by $H^1$ continuity
of $u(\cdot,t)$ and of $\K_\e$, we must have $\K_\e(u(\cdot,t_1)) = 0$
for some $t_1 \in I_{max}$, contradicting $m^* < m_{\e,\om}$.   
\begin{lemma}  
\begin{enumerate}
\item 
$m_{\e,\om} \leq m_{0,0}$, and~\eqref{inf} admits a minimizer if
$m_{\e,\om} < m_{0,0}$;
\item
we have
\begin{equation} \label{inf2}
  m_{\e,\om} = \inf \{ \mathcal{I}_{\omega}(u) \; |  \; 
  u \in H^1(\R^3) \setminus \{0\}, \mathcal{K}_\eps(u) \leq 0 \},
\end{equation}
and a minimizer for this problem is a minimizer for~\eqref{inf}, and vice versa;
\item
given $0 < m^* < m_{\e,\om}$,
there is $\kappa(m^*) > 0$ such that
\begin{equation} \label{gap}
  u \in \mathcal{A}_{m^*} \; \implies \; 
  \K_\e(u) \geq \kappa(m^*) > 0. 
\end{equation}
\end{enumerate}
\end{lemma}

After the local theory, and in particular the perturbation Lemma~\ref{perturbations}, 
the key analytical ingredient is a {\it profile decomposition}, introduced into
the analysis of critical nonlinear dispersive PDE by~\cite{BahouriGerard, Keraani}.
This version, taken from~\cite{KOPV} (and simplified to the radially-symmetric setting), 
can be thought of as
making precise the lack of compactness in the Strichartz estimates for $\dot H^1(\R^3)$ data,
when the data is bounded in $H^1(\R^3)$:
\begin{lemma}(\cite{KOPV}, Theorem 7.5)   \label{PD}
Let $\{ f_n \}_{n=1}^\infty$ be a sequence of radially symmetric functions,
bounded in $H^1(\R^3)$. Possibly passing to a subsequence, there is 
$J^* \in \{0, 1, 2, \ldots \} \cup \{ \infty \}$ such that for each finite $1 \leq j \leq J^*$
there exist (radially symmetric) `profiles' $\phi^j \in \dot H^1 \setminus \{ 0 \}$,
`scales' $\{ \lambda^j_n \}_{n=1}^\infty \subset (0, 1]$, and
`times'  $\{ t^j_n \}_{n=1}^\infty \subset \R$ satisfying, as $n \to \infty$,
\[
  \lambda_n^j \equiv 1 \; \mbox{ or } \;  \lambda_n^j \to 0,
  \qquad
  t_n^j \equiv 0 \; \mbox{ or } \; t_n^j \to \pm \infty.
\]
If $\lambda^j_n \equiv 1$ then additionally $\phi^j \in L^2(\R^3)$.
For some $0 < \theta < 1$, define
\[
  \phi^j_n(x) := \left\{ \begin{array}{cc}
  \left[ e^{i t_n^j \Delta} \phi^j \right] (x) & \lambda_n^j \equiv 1 \\
  (\lambda^j_n)^{-\frac{1}{2}} \left[ e^{i t_n^j \Delta} P_{\geq (\lambda_n^j)^\theta} 
  \phi^j \right] \left( \frac{x}{\lambda^j_n} \right) & \lambda_n^j \to 0, 
  \end{array} \right.
\] 
where $P_{\geq N}$ denotes a standard smooth Fourier multiplier operator
(Littlewood-Paley projector) which removes the Fourier frequencies $\leq N$.  
Then for each finite $1 \leq J \leq J^*$ we have the decomposition
\[
  f_n = \sum_{j=1}^J \phi_n^j + w_n^J
\]
with:
\begin{itemize}
\item
small remainder: 
$\lim\limits_{J \to J^*} \limsup\limits_{n \to \infty} \| e^{i t \Delta} w_n^J \|_{L^{10}_t L^{10}_x (\R)}
= 0$
\item
decoupling: for each $J$,
$\lim\limits_{n \to \infty} \left[ \M(f_n) - \sum\limits_{j=1}^J \M(\phi_n^j)
- \M(w_n^J) \right] = 0$,
and the same statement for the functionals $\E_{\e}$, $\K_\e$,
$\S_{\om,\e}$ and $\mathcal{I}_\om$; 
\item
orthogonality:
$\lim\limits_{n \to \infty} \left[ \frac{\la_n^j}{\la_n^k} + \frac{\la_n^k}{\la_n^j} 
+ \frac{|t_n^j (\la_n^j)^2 - t_n^k (\la_n^k)^2|}{\la_n^j \la_n^k}
\right] = \infty$ for $j \not= k$.
\end{itemize}
\end{lemma}

The global existence and scattering statement 1 of Theorem~\ref{dichotomy}
is established by a contradiction argument. For $0 < m < m_{\e,\om}$, set
\[
  \tau(m) := \sup \left \{ \| u \|_{L^{10}_t L^{10}_x(I_{max})} \; | \; 
  \S_{\eps,\om}(u_0) \leq m, \; \K_\e(u_0) > 0 \right \}
\]
where the supremum is taken over all radially-symmetric solutions of~\eqref{NLS}
whose data $u_0$ satisfies the given conditions.
It follows from the local theory above that $\tau$ is non-decreasing, continuous function 
of $m$ into $[0,\infty]$, and that $\tau(m) < \infty$ for sufficiently small $m$
(by part 4 of Lemma~\ref{local}).
By parts 2-3 of Lemma~\ref{local}, if $\tau(m) < \infty$ for all $m < m_{\e,\om}$,
the first statement of  Theorem~\ref{dichotomy} follows. So we suppose this is {\it not}
the case, and that in fact
\[
  m^* := \sup \{ m \; | \: 0 < m < m_{\eps,\om}, \; \tau(m) < \infty \} < m_{\eps,\om}.
\]
By continuity, $\tau(m^*) = \infty$, and so there exists a sequence $u_n(x,t)$ of
{\it global}, radially-symmetric solutions of~\eqref{NLS} satisfying
\begin{equation} \label{sublevel}
  \S_{\e,\om}(u_n) \leq m^*, \qquad \K_\e(u_n(\cdot,0)) > 0,
\end{equation}
and
\begin{equation} \label{L10}
  \lim_{n \to \infty} \| u_n \|_{L^{10}_t L^{10}_x([0,\infty))} = 
  \lim_{n \to \infty} \| u_n \|_{L^{10}_t L^{10}_x((-\infty,0])} = \infty
\end{equation}
(the last condition can be arranged by time shifting, if needed).
The idea is to pass to a limit in this sequence in order to obtain a solution
sitting at the threshold action $m^*$.
\begin{lemma} \label{Palais-Smale}
There is a subsequence (still labelled $u_n$) such that $u_n(x,0)$ converges in $H^1(\R^3)$.
\end{lemma} 
\begin{proof}
This is essentially Proposition 9.1 of~\cite{KOPV}, with slight modifications to 
incorporate the variational structure. We give a brief sketch.
The sequence $u_n(\cdot,0)$ is bounded in $H^1$ by~\eqref{H1bound}, so
we may apply the profile decomposition Lemma~\ref{PD}: up to subsequence,
\[
  u_n(\cdot,0) = \sum_{j=1}^J \phi_n^j + w_n^J.
\]
If we can show there is only one profile ($J^*=1$), that
$\la_n^1 \equiv 1$, $t_n^1 \equiv 0$, and that $w_n^1 \to 0$ in $H^1$,
we have proved the lemma. 
By~\eqref{sublevel} and the decoupling,
\[
\begin{split}
  m^* - \frac{2}{3(p-1)} \kappa(m^*) & \geq  \S_{\e,\om}(u_n(\cdot,0)) - \frac{2}{3(p-1)} \K_\e(u_n(\cdot,0)) \\
  &= \mathcal{I}_\om(u_n(\cdot,0)) 
  = \sum_{j=1}^J \mathcal{I}_\om(\phi^j_n) + \mathcal{I}_\om(w_n^J)
  + o(1), 
\end{split}
\]
and since $\mathcal{I}_\om$ is non-negative,
we have, for $n$ large enough, 
$\mathcal{I}_\om(\phi^j_n) < m^*$ for each $j$ and
$\mathcal{I}_\om(w_n^J) < m^*$.
Since $m^* < m_{\e,\om}$, it follows from~\eqref{inf2} that 
$\K_\e(\phi^j_n) > 0$ and $\K_\e(w_n^J) \geq 0$, 
so also $\S_{\e,\om}(\phi^j_n) > 0$ and  $\S_{\e,\om}(w_n^J) \geq 0$.
Hence if there is more than one profile, by the decoupling 
\[
  m^* \geq  \S_{\e,\om}(u_n(\cdot,0))
  = \sum_{j=1}^J \S_{\e,\om}(\phi^j_n) + \S_{\e,\om}(w_n^J)
  + o(1), 
\]
we have, for each $j$, and $n$ large enough, for some $\eta > 0$,
\begin{equation} \label{BelowThresh}
  \S_{\e,\om}(\phi^j_n) \leq m^*-\eta, \quad
  \K_\e(\phi^j_n) > 0.
\end{equation}
Following~\cite{KOPV}, we introduce {\it nonlinear profiles} $v_n^j$ associated
to each $\phi_n^j$. 

First, suppose $\la_n^j \equiv 1$. If $t_n^j \equiv 0$, then
$v_n^j = v^j$ is defined to be the solution to~\eqref{NLS} with initial data
$\phi^j$. If $t_n^j \to \pm \infty$, $v^j$ is defined to be the solution scattering 
(in $H^1$) to $e^{i t \Delta} \phi^j$ as $t \to \pm \infty$, and
$v^j_n(x,t) := v^j(t + t_n^j)$. In both cases, it follows from~\eqref{BelowThresh}
that $v^j_n$ is a global solution, with 
$\| v^j_n \|_{L^{10}_t L^{10}_x (\R)} \leq \tau(m^*-\eta) < \infty$.

For the case $\la_n^j \to 0$, we simply let $v_n^j$ be the solution of~\eqref{NLS}
with initial data $\phi^j_n$. 
As in \cite{KOPV} Proposition 8.3, $v_n^j$ is approximated by the solution 
$\tilde{u}_n^j$ of the {\it unperturbed} critical NLS~\eqref{NLS0} (since the profile is concentrating, the sub-critical perturbation `scales away') with data $\phi^j_n$
(or by a scattering procedure in case $t_j^n \to \pm \infty$).
The key additional point here is that by~\eqref{BelowThresh}, and since
$m^* < m_{\e,\om} \leq m_{0,0}$, it follows that for $n$ large enough
\[
  \E_{0}(v_n^j) \leq m^* < m_{0,0}, \quad \K_0(v_n^j) > 0,
\]
and so by~\cite{Kenig}, $\tilde{u}_n^j$ is a {\it global} solution of~\eqref{NLS0},
with $\| \tilde{u}_n^j \|_{L^{10}_t L^{10}_x (\R)} \leq C(m^*) < \infty$.
It then follows from Lemma~\ref{perturbations} that the same is true of
$v_n^j$.

These nonlinear profiles are used to construct what are shown in~\cite{KOPV}
to be increasingly accurate (for sufficiently large $J$ and $n$)
approximate solutions in the sense of Lemma~\ref{perturbations},
\[
  u_n^J(x,t) := \sum_{j=1}^J v_n^j(x,t) + e^{i t \Delta} w_n^J
\]
which are moreover global with uniform space-time bounds.   
This contradicts~\eqref{L10}.

Hence there is only one profile: $J^*=1$, and the decoupling also
implies $\| w^1_n \|_{H^1} \to 0$.
Finally, the possibilities $t^1_n \to \pm \infty$ or $\la^1_n \to 0$
are excluded just as in~\cite{KOPV}, completing the argument.
\end{proof}

Given this lemma, let $u_0 \in H^1(\R^3)$ be the $H^1$ limit of (a subsequence) of $u_n(x,0)$,
and let $u(x,t)$ be the corresponding solution of~\eqref{NLS} on its maximal existence
interval $I_{max} \ni 0$. 
We see $\S_{\e,\om}(u) = \S_{\e,\om}(u_0) \leq m^*$.
Whether $u$ is global or not, it follows from
Lemma~\ref{local} (part 2), \eqref{L10} and Lemma~\ref{perturbations}, that
\[
  \| u \|_{L^{10}_t L^{10}_x (I_{max})} = \infty, \quad
  \mbox{ hence also } \S_{\e,\om}(u) = m^*.
\]
It follows also that 
\[
  \{ u(\cdot,t) \; | \;  t \in I_{max} \} \mbox{ is a pre-compact set in } H^1(\R^3).
\] 
To see this, let $\{ t_n \}_{n=1}^\infty \subset I_{max}$, and note that since
\[
  \| u \|_{L^{10}_t L^{10}_x ((-T_{min}, t_n])} = 
  \| u \|_{L^{10}_t L^{10}_x ([t_n, T_{max}))} = \infty,
\]
and so (the proof of) Lemma~\ref{Palais-Smale} applied to the sequence
$u(x,t-t_n)$ implies that $\{ u(x,t_n) \}$ has a convergent subsequence in $H^1$. 

The final step is to show that this `would-be' solution $u$ with these special properties,
sometimes called a {\it critical element} cannot exist.
For this, first note that $u$ must be global: $I_{max} = \R$. This is because if, say,
$T_{max} < \infty$, then for any $t_n \to T_{max}-$, $u(\cdot,t_n) \to \tilde{u}_0 \in H^1(\R^3)$
(up to subsequence) in $H^1$, by the pre-compactness. Then by comparing $u$ with 
the solution $\tilde{u}$ of~\eqref{NLS} with initial data $\tilde{u}_0$ at $t = t_n$ using 
Lemma~\eqref{perturbations}, we conclude that $u$ exists for times beyond $T_{max}$,
a contradiction. 

Finally, the possible existence of (the now global) solution 
$u$ is ruled out via a suitably cut-off version of the virial identity~\eqref{virial},
using~\eqref{gap}, and the compactness to control the 
errors introduced by the cut-off, exactly as in \cite{KOPV}
(Proposition 10, and what follows it).
$\Box$


\section*{Acknowledgements}
The authors thank T.-P.Tsai for suggesting the problem, helpful discussion, and sharing the pre-print \cite{Gust}.
The first author acknowledges support from the NSERC CGS. 
Research of the second author is supported by an NSERC Discovery Grant.


\noindent email: colesmp@math.ubc.ca \newline 
email: gustaf@math.ubc.ca \newline
\newline
Keywords: nonlinear Schr\"odinger equation, Lyapunov-Schmidt reduction, ground state solitary waves, scattering, blow-up

\end{document}